\documentclass[11pt]{article}

\usepackage[a4paper,margin=1in]{geometry}
\usepackage{amsmath,amssymb,amsthm,mathtools}
\usepackage{mathrsfs}
\usepackage{xcolor}
\usepackage{hyperref}
\usepackage[nameinlink,capitalize]{cleveref}
\usepackage[titletoc,title]{appendix}
\usepackage{enumerate}
\usepackage{authblk}
\usepackage{booktabs}

\hypersetup{
  colorlinks=true,
  linkcolor=blue,
  citecolor=blue,
  urlcolor=blue
}
\usepackage{tikz}
\usetikzlibrary{arrows.meta,calc,decorations.pathreplacing}

\numberwithin{equation}{section}
\numberwithin{figure}{section}
\numberwithin{table}{section}

\newtheorem{theorem}{Theorem}[section]
\newtheorem{proposition}[theorem]{Proposition}
\newtheorem{lemma}[theorem]{Lemma}

\newtheorem{assumption}[theorem]{Assumption}
\newtheorem{remark}[theorem]{Remark}

\newcommand{\R}{\mathbb R}
\newcommand{\C}{\mathbb C}
\newcommand{\eps}{\varepsilon}
\newcommand{\dd}{\,\mathrm d}

\newcommand{\one}{\mathbf 1}
\newcommand{\dist}{\operatorname{dist}}
\newcommand{\loc}{\operatorname{loc}}
\newcommand{\Span}{\operatorname{span}}

\newcommand{\Ker}{\operatorname{Ker}}
\newcommand{\Newt}{\mathrm{Newt}}
\newcommand{\icap}{\operatorname{Cap}}

\title{Interface-Shifted Minnaert Resonances}
\author[1]{Huaian Diao\thanks{E-mail: \texttt{diao@jlu.edu.cn}}}
\author[2]{Long Li\thanks{E-mail: \texttt{long.li@ricam.oeaw.ac.at}}}
\author[2]{Mourad Sini\thanks{E-mail: \texttt{mourad.sini@oeaw.ac.at}}}
\affil[1]{School of Mathematics, Jilin University, Changchun 130012, China}
\affil[2]{RICAM, Austrian Academy of Sciences, Altenberger Stra\ss e 69, 4040 Linz, Austria}

\date{}

\begin{document}

\maketitle

\begin{abstract}
\noindent We develop a unified asymptotic theory for Minnaert resonances
generated by a small bubble \(D_\varepsilon\) placed near an
interface \(\Gamma\) of material discontinuity in an acoustic
background. The bubble has characteristic size \(\varepsilon>0\),
with its density and bulk modulus depending on the same small
parameter through the Minnaert high-contrast scaling.
The analysis is organized by the scaled separation
\[
\Theta_\varepsilon
:=
\frac{\operatorname{dist}(D_\varepsilon,\Gamma)}{\varepsilon},
\]
and covers the three natural regimes
\[
\Theta_\varepsilon\to+\infty,
\qquad
\Theta_\varepsilon\to\Theta_*\in(0,\infty),
\qquad
\Theta_\varepsilon\to0:
\]
\begin{enumerate}
\item In the far-separation regime, the interface is invisible at leading order and the classical full-space Minnaert law is recovered. 

\item In the critical regime, the rescaled geometry retains both the bubble and the interface at finite separation, and the Newtonian capacitance is replaced by a flat-interface capacitance determined by the limiting tangent transmission problem. 

\item In the near-contact regime, this family converges to a contact-limiting capacitance, yielding the same leading-order
Minnaert resonance for all gap scalings satisfying \(\operatorname{dist}(D_\varepsilon,\Gamma) = o(\varepsilon)\).
\end{enumerate}

\noindent These three regimes provide a unified description of how the resonance frequency changes as the bubble approaches the interface, from the classical isolated-bubble limit to the contact limit. The analysis relies on a uniform reduction of the resonance problem to a scalar equation and on a variational characterization of the interface-dependent capacitances; in particular, Mosco convergence of the associated transmission energies yields convergence of the capacitary potentials and minimum energies up to the singular contact configuration.
\end{abstract}

\textbf{Keywords}:  Minnaert resonances; interface of material discontinuity; uniform asymptotics; near-contact; flat-interface capacitance.

\section{Introduction}

The Minnaert resonance is the classical mechanism through which a gas bubble whose diameter is much smaller than the acoustic wavelength can nevertheless generate a large scattered field.  Since Minnaert's original physical description \cite{Minnaert1933}, this phenomenon has become a central example of subwavelength resonance.  Its mathematical analysis has been developed through layer-potential methods, resonant asymptotics, pole characterizations and, more recently, uniform resolvent estimates; see for instance \cite{AZ-18, AmmariKangLee2009,DabrowskiGhandricheSini2021, MPS,LiSini2026}. These works, show that in a homogeneous or locally smooth background the leading Minnaert frequency is determined by the bubble contrast and by a static capacitance of the reference shape.

The common structural assumption behind the classical formula is that the background is smooth at the scale of the bubble.  After rescaling around the bubble, the background Green function then converges to the free-space Newtonian kernel, up to the local value of the material coefficients.  This explains why the usual capacitance is the leading object in homogeneous media, smoothly varying media, and smooth-background imaging models using bubbles as contrast agents \cite{DabrowskiGhandricheSini2021}.  It also explains the form of the point-scatterer limits and resolvent expansions obtained in the homogeneous case \cite{LiSini2026}.

The situation changes qualitatively when the bubble approaches a
material interface. In this case, the local free-space approximation
underlying the classical Minnaert law is no longer uniform with
respect to the bubble--interface distance. When this distance becomes
comparable with the bubble size, the interface remains visible after
rescaling, and the leading resonance is no longer governed solely by
the classical Newtonian capacitance. Instead, a two-phase transmission
capacitance enters the leading-order law.

The influence of nearby interfaces on the oscillation and resonance frequencies of gas bubbles has been investigated in the physical-acoustics literature, mainly for spherical bubbles and special geometries such as planar fluid--fluid interfaces, rigid boundaries, or layered media; see, for instance, \cite{DoinikovAiredBouakaz2011,MaksimovPolovinka2017,MaksimovPolovinka2018,Maksimov2018}. These works rely on explicit coordinate representations, symmetry arguments, or formal acoustic approximations and predict a distance-dependent modification of the classical Minnaert frequency. Such interface-induced resonance shifts have also recently been observed experimentally at the single-bubble level \cite{BouchetEtAl2024}. To the best of our knowledge, however, a rigorous scattering-resonance analysis for a small bubble of general shape embedded in a discontinuous acoustic background, uniformly with respect to the bubble--interface distance and including the near-contact regime, has not previously been available. The purpose of the present work is to provide such an analysis and to show that the regimes of large, finite, and vanishing scaled distance are governed by a single family of interface-dependent capacitances.

To describe this transition, we introduce the scaled separation
\[
\Theta_\varepsilon
:=
\frac{\operatorname{dist}(D_\varepsilon,\Gamma)}{\varepsilon},
\]
where \(D_\varepsilon\) denotes a bubble of characteristic size \(\eps\) the material
interface.
There are three natural asymptotic regimes:
\[
\Theta_\varepsilon\to+\infty,
\qquad
\Theta_\varepsilon\to\Theta_*\in(0,\infty),
\qquad
\Theta_\varepsilon\to0.
\]
They correspond respectively to far separation, finite scaled
separation, and near contact. The main purpose of this paper is to
treat these three cases within a single asymptotic framework, without
requiring the scaled distance to remain bounded away from zero.

Our main result Theorem \ref{th:main} shows that, for all sufficiently small
\(\varepsilon\), there exist exactly two Minnaert resonances
\(z_\varepsilon^\pm\). They satisfy the uniform expansion
\[
z_\varepsilon^\pm
=
\pm
\left(
\frac{\bar k_b}{|B|}
\frac{\icap^{\mathrm{phys}}(D_\varepsilon)}
{\varepsilon}
\right)^{1/2}
+
O(\varepsilon),
\]
where the remainder is uniform with respect to the separation regimes
considered here. Thus the geometry enters the leading resonance law
through the normalized physical capacitance
\[
\frac{\icap^{\mathrm{phys}}(D_\varepsilon)}
{\varepsilon}.
\]

The asymptotic behavior of this normalized capacitance is described by
a one-parameter family of flat-interface capacitances
\[
\icap^{\mathrm{flat}}_{\Theta,\nu}(B),
\qquad
\Theta\in[0,\infty].
\]
If \(\Theta_\varepsilon\to+\infty\), the interface disappears from the
blow-up limit and the classical homogeneous-phase Minnaert law is
recovered. If
\(\Theta_\varepsilon\to\Theta_*\in(0,\infty)\), the rescaled bubble
remains at a finite distance from the limiting tangent interface, and the leading-order resonance is determined by the corresponding flat-interface capacitance.
If \(\Theta_\varepsilon\to0\), the rescaled gap vanishes and the flat comparison problem approaches its contact limit. The leading-order resonance therefore converges to that
determined by the contact-limit capacitance. See Figure \ref{fig:flat-comparison-geometry} for the associated flat comparison geometry.

The proof builds on a variational approach previously used to
characterize Minnaert-type subwavelength resonances in
time-modulated media~\cite{FH-04} and subsequently applied to
Fabry--P\'erot resonances~\cite{LiSinifabry}.
Projection onto the zero Neumann mode and its orthogonal
complement leads to a scalar equation for the Minnaert
resonances.
The main analytical difficulty here is to control this
reduction as the bubble shrinks and approaches the transmission
interface, without imposing a positive lower bound on the
scaled distance $\Theta_\eps$.
We establish uniform solvability estimates for the exterior
problem on compact spectral sets avoiding the resonances of
the transmission problem without the bubble.
These estimates provide uniform control of the remainder in
the scalar reduction, yielding exactly two Minnaert resonances, with leading terms determined
by the normalized physical capacitance.

Identifying this leading term further requires passing to a
limiting transmission geometry in which the interface may
touch the reference bubble.
After rescaling to a fixed reference domain, we encode the
varying interface in the material coefficients and verify
their almost-everywhere convergence.
Mosco’s stability theorem then yields strong convergence of the capacitary potentials and hence convergence of the associated minimum energies, including in the contact configuration.

The paper is organized as follows. Section \ref{sec:m} introduces the geometric
and material setting, defines the physical and flat-interface
capacitances, and states the main results. Section \ref{sec:Ch} develops the
characterization of scattering resonances in the physical and rescaled
domains and identifies the Minnaert resonances. Section \ref{sec:p} is devoted
to the proof of the main theorem. The monotonicity and boundary integral formulation of the flat capacitance, the technical uniform solvability, and
capacity-continuity arguments are collected in the appendices.

\section{Setting and main results} \label{sec:m}

\subsection{Geometric and material setting}
Let $\Omega\subset\R^3$ be a bounded, and connected open set with $C^2$ boundary and connected exterior, and set
\[
\Gamma:=\partial\Omega,
\qquad
\Omega_+:=\Omega,
\qquad
\Omega_-:=\R^3\setminus\overline\Omega.
\]
Let \(\Gamma\subset\mathbb R^3\) be a \(C^2\) interface separating
\(\Omega_-\) and \(\Omega_+\), and let \(\nu_\Gamma\) denote the unit
normal pointing from \(\Omega_+\) into \(\Omega_-\), and define the signed-distance function \(s_{\Gamma}: \mathbb R^3 \to \mathbb R\) by
\[
s_{\Gamma}(x)
:=
\begin{cases}
 \operatorname{dist}(x,\Gamma),
 & x\in\Omega_+,
 \\[1mm]
 0,
 & x\in\Gamma,
 \\[1mm]
 -\operatorname{dist}(x,\Gamma),
 & x\in\Omega_-.
\end{cases}
\]
Since \(\Gamma\) is of class \(C^2\), there exists \(r_\Gamma>0\) such
that \(s_\Gamma\in C^2(U_{r_\Gamma})\), where
\begin{align}
U_{r_\Gamma}
:=
\bigl\{x\in\mathbb R^3:
\operatorname{dist}(x,\Gamma)<r_\Gamma\bigr\}. \notag
\end{align}
Moreover, every \(x\in U_{r_\Gamma}\) has a unique nearest point
\(\pi_\Gamma(x)\in\Gamma\), and admits the below representation
\begin{align} \label{eq:2}
x
=
\pi_\Gamma(x)
-
s_\Gamma(x)\nu_\Gamma\bigl(\pi_\Gamma(x)\bigr),
\end{align}
see \cite[Lemma~14.16 and its proof]{GT_11}.

Having described the background interface, we now introduce the family of shrinking inclusions. Let \(B\subset\mathbb R^3\) be a bounded connected open set containing the origin, with a
\(C^{2}\) boundary and a connected
exterior \(\mathbb R^3\backslash \overline B\). For each
\(\varepsilon>0\), let
\[
D_\varepsilon=x_\varepsilon + \varepsilon B,
\qquad
\overline{D_\varepsilon}\subset\Omega_+,
\]
where \(x_\varepsilon\) specifies the position of the inclusion, 
dependent on \(\varepsilon\). To describe the position of \(D_\varepsilon\) relative to
\(\Gamma\), set
\[
d_\varepsilon
:=
\operatorname{dist}(D_\varepsilon,\Gamma).
\]
Since \(\partial D_\varepsilon\) and \(\Gamma\) are compact, there exist \(e_\varepsilon\in\partial D_\varepsilon\) and
\(p_\varepsilon\in\Gamma\) such that
\begin{align}\label{eq:mini}
|e_\varepsilon-p_\varepsilon|=d_\varepsilon.
\end{align}
We then define
\begin{align} \label{eq:scaled}
\xi_\varepsilon:=\frac{e_\varepsilon-x_\varepsilon}{\varepsilon}
\in\partial B,
\end{align}
so that
\[
e_\varepsilon=x_\varepsilon+\varepsilon\xi_\varepsilon,
\]
and define
\begin{align} \label{eq:21}
\Theta_\eps:=\frac{d_\varepsilon}{\varepsilon}.
\end{align}

\begin{figure}[htbp]
\centering
\begin{tikzpicture}[
    x=1cm,
    y=1cm,
    >=Latex,
    line cap=round,
    line join=round,
    every node/.style={font=\small}
]

\filldraw[
    fill=blue!8,
    draw=black,
    line width=1.0pt
]
(0,0) ellipse [x radius=4.4, y radius=2.7];

\node at (-1.55,0.35) {$\Omega_+=\Omega$};
\node at (3.65,-3.25) {$\Omega_-$};
\node[anchor=west] at (4.02,-0.95)
{$\Gamma=\partial\Omega$};

\coordinate (xc) at (0,-1.42);
\def\r{0.62}

\filldraw[
    fill=green!14,
    draw=black,
    line width=0.95pt
]
(xc) circle [radius=\r];

\fill (xc) circle (1.25pt);
\node[above right=1pt] at (xc) {$x_\varepsilon$};

\node[right=7pt] at ($(xc)+(\r,0)$)
{$D_\varepsilon$};

\coordinate (p) at (0,-2.7);

\coordinate (y) at (0,-2.04);

\fill (p) circle (1.35pt);
\fill (y) circle (1.35pt);

\node[below left=1pt] at (p)
{$p_\varepsilon$};

\node[right=5pt] at (y)
{$x_\varepsilon+\varepsilon\xi_\varepsilon$};

\draw[
    densely dashed,
    line width=0.8pt
]
(y) -- (p);

\node[left=5pt] at ($(y)!0.50!(p)$)
{$d_\varepsilon$};

\draw[
    -{Latex[length=2.5mm]},
    line width=0.95pt
]
(p) -- ++(0,-0.85);

\node[right=4pt] at ($(p)+(0,-0.52)$)
{$\nu_\Gamma(p_\varepsilon)$};

\end{tikzpicture}

\caption{
Near-contact geometry in the physical variables.}
\label{fig:physical-near-contact}
\end{figure}
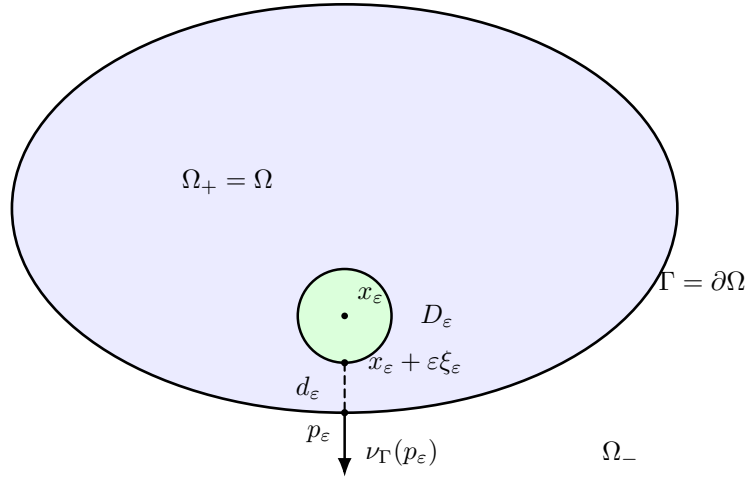
The geometric configuration of \(D_\varepsilon\) relative to the interface \(\Gamma\), including the near-contact case, is illustrated in Figure \ref{fig:physical-near-contact}.

\paragraph{Geometric assumption} We distinguish the following two geometric regimes.
\begin{assumption}\label{as_1}
\begin{enumerate}
\item[(i)]  In the finite scaled-distance regime, we assume that there exists
\(p_*\in\Gamma\) such that
\begin{align} \label{eq:asump_1}
\lim_{\eps \rightarrow 0} x_\varepsilon = p_*,
\qquad
\lim_{\eps \rightarrow 0} \Theta_\eps
=
\Theta_*\in[0,\infty).
\end{align}
Set
\[
\nu_*:=\nu_\Gamma(p_*).
\]
We further assume that \(B\) has a unique supporting point in the
direction \(\nu_*\), namely, that there exists a unique
\(\xi_*\in\partial B\) such that
\[
\xi_*
=
\operatorname*{arg\,max}_{\xi\in\overline B}
\nu_*\cdot\xi.
\]

\item[(ii)] In the separated regime,
\[
\lim_{\eps \rightarrow 0} \Theta_\eps
= +\infty.
\]
\end{enumerate}
\end{assumption}
The background density and bulk modulus are piecewise constant across
$\Gamma$:
\[
\rho^0(x)
=
\begin{cases}
\rho_+, & x\in\Omega_+,\\[1mm]
\rho_-, & x\in\Omega_-,
\end{cases}
\qquad
k^0(x)
=
\begin{cases}
k_+, & x\in\Omega_+,\\[1mm]
k_-, & x\in\Omega_-,
\end{cases}
\]
where
\(
\rho_\pm>0\) and \(k_\pm>0\)
are constants.
Inside the bubble, we impose the Minnaert scaling
\[
\rho_b^\eps:=\bar\rho_b\eps^2,
\qquad
k_b^\eps:=\bar k_b\eps^2,
\qquad
\bar\rho_b,\;\bar k_b>0.
\]
The full coefficients are therefore
\[
\rho_\eps(x)
=
\begin{cases}
\rho_b^\eps,
& x\in D_\eps,\\[1mm]
\rho_+,
& x\in\Omega_+\setminus\overline{D_\eps},\\[1mm]
\rho_-,
& x\in\Omega_-,
\end{cases}
\qquad
k_\eps(x)
=
\begin{cases}
k_b^\eps,
& x\in D_\eps,\\[1mm]
k_+,
& x\in\Omega_+\setminus\overline{D_\eps},\\[1mm]
k_-,
& x\in\Omega_-.
\end{cases}
\]
The $\eps^2$ scaling of the mass density or the bulk modulus mentioned above ensures that the generated subwavelength resonance is of order $O(1)$, a key feature in metamaterial applications, as it enables the manipulation of moderate frequencies to design effective dispersive media. We refer to \cite{AFZ-17, AZ-17,C-M-P-T-1} for the time harmonic case and to \cite{C-M-P-T-2,MS-241, MS-242} for the time domain case.

\paragraph{High-contrast Hamiltonian} We consider the associated Hamiltonian $H_{\rho_\eps, k_\eps}$ defined by 
\begin{align}
\left(H_{\rho_\eps, k_\eps}\right) u := -{k_\eps}\nabla \cdot \frac{1}{\rho_\eps} \nabla u \notag
\end{align}
with the domain 
\begin{align}
{\mathcal D}(H_{\rho_\eps,k_\eps}):= \bigg\{u \in H^1(\R^3),\; k_\eps \nabla\cdot \frac{1}{\rho_\eps}\nabla u \in L^2(\R^3) \bigg\}. \notag
\end{align}
It falls within the class of black-box Hamiltonians \cite{LSW}. Choose \(\chi\in C_c^\infty(\mathbb R^3)\) equal to one
on a neighbourhood of \(\Omega_+\). For each fixed \(\varepsilon>0\),
the cut-off outgoing resolvent
\[
    \chi\bigl(H_{\rho_\varepsilon,k_\varepsilon}-z^2\bigr)^{-1}\chi,
    \qquad \operatorname{Im}z>0,
\]
extends meromorphically to \(z\in\mathbb C\) as an
\(L^2(\mathbb R^3)\)-operator-valued function; see \cite{DM}. Its poles, which are independent of the choice of such a cut-off, are called the scattering
resonances of \(H_{\rho_\varepsilon,k_\varepsilon}\).

Let \(H_{\mathrm{bg}}\) denote the background transmission
Hamiltonian obtained by restoring the background
coefficients \(\rho^0\) and \(k^0\) in \(D_\varepsilon\).
Throughout the paper, we set
\[
    \mathcal O:=\mathbb C\setminus\mathcal O_{\mathrm{bg}}.
\]
Here, \(\mathcal O_{\mathrm{bg}}\) denotes the set of scattering
resonances of \(H_{\mathrm{bg}}\) in the complex
variable \(z\), defined as the poles of the meromorphically
continued cut-off outgoing resolvent associated with
\((H_{\mathrm{bg}}-z^2)^{-1}\). We refer to \cite{J19, PV-99, P-V:1919} for results on the distribution of these resonances.

\subsection{Flat comparison geometry and capacitances}

Before stating the result, we introduce the two capacitance quantities
that enter the asymptotic formulas: the physical capacitance of
\(D_\varepsilon\) in the original transmission medium and its
flat-interface counterpart associated with the reference domain \(B\).

Set
\[
D^{1,2}(\mathbb R^3)
:=
\left\{
v\in L^6(\mathbb R^3):
\nabla v\in L^2(\mathbb R^3)
\right\},
\]
and, for any bounded Lipschitz domain \(E\subset\mathbb R^3\), define
\[
\mathcal A_E
:=
\left\{
v\in D^{1,2}(\mathbb R^3):
v=1 \quad\text{a.e. in }E
\right\}.
\]
We first define the physical capacitance of \(D_\varepsilon\) by
\begin{align}\label{eq:phy}
\icap^{\mathrm{phys}}(D_\varepsilon)
:=
\inf_{v\in\mathcal A_{D_\varepsilon}}
\int_{\mathbb R^3\setminus\overline{D_\varepsilon}}
\left(\rho_+^{-1}\mathbf 1_{\Omega_+}(x)
+
\rho_-^{-1}\mathbf 1_{\Omega_-}(x)\right)|\nabla v(x)|^2\,\dd x.
\end{align}
This definition is the weighted whole-space analogue of the
classical variational definition of Newtonian capacity; see,
e.g., \cite{DePhilippisMariniMukoseeva2021} for the unweighted case.

The quantity entering the leading-order Minnaert law is its normalized
version
\[
\mathfrak C_\varepsilon
:=
\frac{1}{\varepsilon}
\icap^{\mathrm{phys}}(D_\varepsilon).
\]
Its limiting behavior is described by the flat-interface capacitance
associated with the flat
comparison interface \(\Pi_{\Theta,\nu}^{B}\). Here
\begin{align}\label{eq:interface}
\Pi_{\Theta,\nu}^{B}
:=
\left\{
{\widetilde x}\in\mathbb R^3:
\nu\cdot {\widetilde x}=\ell_B(\nu)+\Theta
\right\}, \qquad 0\leq\Theta<\infty,
\end{align}
where 
\[
\ell_B(\nu)
:=
\max_{{\widetilde x}\in\overline B}\nu\cdot {\widetilde x}, \quad \nu \in \mathbb S^2.
\]
For
\(\Theta\in[0,\infty)\), define
\begin{align}
\icap^{\mathrm{flat}}_{\Theta,\nu}(B)
:=
\inf_{v\in\mathcal A_B}
\int_{\mathbb R^3\setminus\overline B}
a_{\Theta,\nu}^{\mathrm{flat}}({\widetilde x})
|\nabla v({\widetilde x})|^2\,d{\widetilde x}. \notag
\end{align}
Here,
\begin{align}\label{eq:flat_1}
a_{\Theta,\nu}^{\mathrm{flat}}({\widetilde x})
:=
\rho^{-1}_+
\mathbf 1_{\{\nu\cdot {\widetilde x}<\ell_B(\nu)+\Theta\}}({\widetilde x})
+
\rho^{-1}_{-}
\mathbf 1_{\{\nu\cdot {\widetilde x}>\ell_B(\nu)+\Theta\}}({\widetilde x}), \quad \Theta\in[0,\infty].
\end{align}
Here, we refer to \(\icap^{\mathrm{flat}}_{0,\nu}(B)\)
as the contact-limit capacitance, in view of the continuity of the
flat-interface capacitance at \(\Theta=0\); see the final part of the
proof of Proposition \ref{pr:fc}.

At the endpoint \(\Theta=\infty\), the comparison interface leaves
every fixed bounded subset of the rescaled space. Accordingly, set
\begin{align}\label{eq:flat_2}
&a_{\infty,\nu}^{\mathrm{flat}}({\widetilde x})
:=
\rho^{-1}_+,
\qquad
{\widetilde x}\in\mathbb R^3,
\end{align}
and hence
\[
\icap^{\mathrm{flat}}_{\infty,\nu}(B)
=
\rho^{-1}_+\icap_{\Newt}(B),
\]
where
\begin{align}
\icap_{\Newt}(B)
:=
\inf_{v\in\mathcal A_B}
\int_{\mathbb R^3\setminus\overline B}
|\nabla v({\widetilde x})|^2\,d{\widetilde x}. \notag
\end{align}
Although the endpoint coefficient is independent of \(\nu\), we retain
\(\nu\) in the notation for uniformity.

The flat-interface capacitance mentioned above also admits the following boundary-integral representation:
\begin{align}\label{eq:flat-capacity-layer}
\icap^{\mathrm{flat}}_{\Theta,\nu}(B)
    =
    \left\langle 1,
        \mathcal S_{\Theta,\nu}^{-1}1
    \right\rangle_{H^{-1/2}(\partial B),H^{1/2}(\partial B)}.
\end{align}
Here, for \(\Theta \in [0,+\infty)\), the operator \(  \mathcal S_{\Theta,\nu}:
    H^{-1/2}(\partial B)\longrightarrow H^{1/2}(\partial B)\) is defined by
\begin{align*}
 (\mathcal S_{\Theta,\nu}\varphi)({\widetilde x})
    :=
    \frac{\rho_+}{4\pi}
    \int_{\partial B}
    \left(
        \frac{1}{|{\widetilde x}-{\widetilde y}|}
        + \frac{1/\rho_+ -  1/\rho_-}{1/\rho_+ + 1/\rho_-}\frac 1{|R_\Theta{\widetilde x}- {\widetilde y}|}
    \right)
    \varphi({\widetilde y})\,dS({\widetilde y}),
    \qquad {\widetilde x}\in\partial B,
\end{align*}
where 
\begin{align}\label{eq:reflect}
    R_\Theta {\widetilde y}
    :=
    {\widetilde y}-2\bigl(\nu\cdot {\widetilde y} - \ell_B(\nu)-\Theta\bigr)\nu
\end{align}
is the reflection of \({\widetilde y} \in \R^3\) across the limiting flat interface \(\Pi^B_{\Theta,\nu}\).
Moreover, as a function of \(\Theta\), the flat capacitance is nondecreasing if \(\rho_+ < \rho_-\), nonincreasing if \(\rho_+> \rho_-\), and constant if \(\rho_+=\rho_-\). The proof of these monotonicity properties, together with a boundary-integral representation of \(\icap^{\mathrm{flat}}_{\Theta,\nu}(B)\), is deferred to Appendix~\ref{app:flat-capacitance}.

\subsection{Main theorem}

We can now state the main theorem.

\begin{theorem}\label{th:main}
Let \(\eps > 0\) and let \(K\subset \mathcal O\) be compact. Assume that Assumption \ref{as_1} holds. There exists $\eps_K$ such that for $\eps \in (0,\eps_K)$, the Hamiltonian $H_{\rho_\eps, k_\eps}$
admits at most two Minnaert resonances \(z_\varepsilon^\pm\) in \(K\), satisfying
\begin{align}\label{eq:main_1}
z_\varepsilon^\pm
=
\pm
\left(
\frac{\bar k_b}{|B|}
\mathfrak C_\varepsilon
\right)^{1/2}
+
z_{\varepsilon,\mathrm{rem}}^\pm ,
\end{align}
where
\begin{align}\label{eq:main_2}
\left|
z_{\varepsilon,\mathrm{rem}}^\pm
\right|
\leq C\varepsilon .
\end{align}
The constant \(C\) is independent of \(\eps\) and \(\operatorname{dist}(D_\varepsilon,\Gamma)\).

Moreover, the normalized physical capacitance $\mathfrak C_\eps$ admits the following limiting descriptions.

\begin{enumerate}[(a)]
\item If
\(
\lim_{\eps \rightarrow 0}\Theta_\varepsilon = +\infty,
\)
then
\[
\lim_{\eps \rightarrow 0} \mathfrak C_\eps =
\rho_+^{-1}\icap_{\mathrm{Newt}}(B),
\]
and hence
\[
z_\varepsilon^\pm
=
\pm
\left(
\frac{\bar k_b}{|B|}
\rho_+^{-1}\icap_{\mathrm{Newt}}(B)
\right)^{1/2}
+o(1), \qquad \mathrm{as}\; \eps \rightarrow 0. 
\]
Thus the classical homogeneous-phase Minnaert law is recovered.

\item If
\(
\lim_{\eps \rightarrow 0}\Theta_\varepsilon = 
\Theta_*\in(0,\infty),
\)
then
\[
\lim_{\eps \rightarrow 0} \mathfrak C_\varepsilon
=
\icap^{\mathrm{flat}}_{\Theta_*,\nu_*}(B).
\]

\item If
\(\lim_{\eps \rightarrow 0}\Theta_\varepsilon = 0, \)
then
\[
\lim_{\eps \rightarrow 0}\mathfrak C_\varepsilon =
\icap^{\mathrm{flat}}_{0,\nu_*}(B).
\]

The corresponding limiting Minnaert frequencies follow immediately
from the expansion above.
\end{enumerate}
\end{theorem}

\begin{remark}
Let \(e_\varepsilon \in \partial D_{\varepsilon}\) and \(p_\varepsilon \in \Gamma\) be the minimizing pair satisfying \eqref{eq:mini}, and let \(\zeta_\varepsilon\) be given by \eqref{eq:scaled}. Then the point \(p_\varepsilon\) is represented by
\[
p_\eps = x_\eps + \eps {\widetilde x}_{p_\eps}
\]
in the rescaled variables
\( {\widetilde x}=\frac{x-x_\varepsilon}{\varepsilon}\). Equivalently, 
\[
{\widetilde x}_{p_\eps} = 
\frac{p_\varepsilon-x_\varepsilon}{\varepsilon}
=
\xi_\varepsilon - \Theta_\varepsilon\nu_\varepsilon,
\]
Hence the tangent plane to the rescaled interface at \({\widetilde x}_{p_\eps}\)
is
\[
\nu_\varepsilon\cdot {\widetilde x}
=
\nu_\varepsilon\cdot\xi_\varepsilon
-
\Theta_\varepsilon.
\]
Under Assumption \ref{as_1} (i), we have
\begin{align}\label{eq:30}
\lim_{\eps \rightarrow 0} p_\varepsilon = p_*,
\qquad
\lim_{\eps \rightarrow 0}  \xi_\varepsilon = \xi_*,
\qquad
\lim_{\eps \rightarrow 0} \nu_\Gamma(p_\varepsilon) = \nu_*,
\end{align}
see Lemma \ref{lem:closest-point-localization} for the proof.
Then, it follows that
the rescaled interface converges locally to \( \Pi_{\Theta_*,\nu_*}^{B}\), defined by \eqref{eq:interface}.
Thus, when \(\Theta_*=0\), the limiting interface is the supporting
plane of \(B\) in the direction \(\nu_*\). When \(\Theta_*>0\), it is
the parallel plane obtained by translating the supporting plane by the
distance \(\Theta_*\) in the direction \(\nu_*\). The flat comparison geometry is illustrated in 
Figure~\ref{fig:flat-comparison-geometry}.

On the other hand, under Assumption \ref{as_1} (ii),
the rescaled interface leaves every fixed bounded subset of
\(\mathbb R^3\).
\end{remark}

\begin{figure}[t]
\centering
\begin{tikzpicture}[
  x=1.15cm,
  y=1.15cm,
  >=Latex,
  every node/.style={font=\small},
  support/.style={
    blue!65!black,
    densely dashed,
    line width=0.9pt
  },
  interface/.style={
    orange!80!black,
    line width=1.05pt
  }
]

\def\h{1.45}

\fill[orange!4]
  (-3.15,\h) rectangle (3.20,2.05);

\draw[support]
  (-3.00,0) -- (3.00,0);

\draw[interface]
  (-3.00,\h) -- (3.00,\h);

\path[
  fill=gray!18,
  draw=black,
  line width=0.9pt
]
  (0,0)
  .. controls (-0.38,-0.02) and (-1.18,-0.28)
     .. (-1.45,-1.02)
  .. controls (-1.68,-1.72) and (-0.88,-2.25)
     .. (0.02,-2.28)
  .. controls (0.98,-2.27) and (1.58,-1.66)
     .. (1.43,-0.93)
  .. controls (1.30,-0.29) and (0.38,-0.02)
     .. (0,0)
  -- cycle;

\node at (0,-1.25) {$B$};

\fill (0,0) circle (1.35pt);
\node[below left=2pt] at (0,0) {$\xi_*$};

\draw[
  -{Latex[length=2.5mm]},
  line width=1.0pt
]
  (0,0.12) -- (0,0.98);

\node[left=3pt] at (0,0.58) {$\nu$};

\draw[<->,line width=0.85pt]
  (2.50,0.05)
  --
  node[right=4pt] {$\Theta_*$}
  (2.50,\h-0.05);

\draw[line width=0.75pt]
  (2.38,0) -- (2.62,0);

\draw[line width=0.75pt]
  (2.38,\h) -- (2.62,\h);

\node[
  anchor=south west,
  text=blue!65!black
]
  at (-2.92,0.05)
  {$\Pi_{0,\nu}^{B}:
    \ \nu\!\cdot\!{\widetilde x}=\ell_B(\nu)$};

\node[
  anchor=south west,
  text=orange!80!black
]
  at (-2.92,\h+0.06)
  {$\Pi_{\Theta,\nu}^{B}:
    \ \nu\!\cdot\!{\widetilde x}=\ell_B(\nu)+\Theta_*$};

\end{tikzpicture}

\caption{
A normal cross-section of the flat comparison geometry.
The dashed line is the supporting plane of \(B\) at \(\xi_*\),
whereas the solid line is the flat comparison interface
\(\Pi_{\Theta,\nu}^{B}\), obtained by translating the supporting
plane by the distance \(\Theta_*\) in the direction \(\nu\).
}
\label{fig:flat-comparison-geometry}
\end{figure}
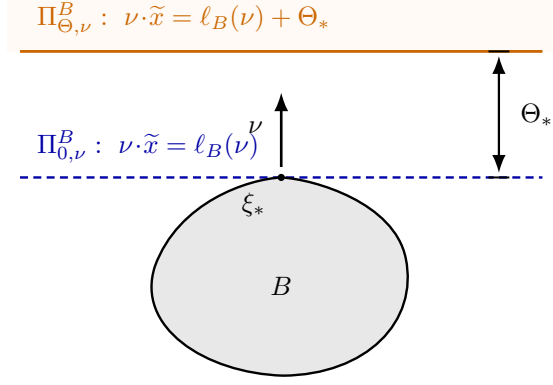

\section{Characterization of the Minnaert resonance}
\label{sec:Ch}

This section introduces the characterization of scattering resonances in the physical and rescaled domains in Section \ref{sec:Ch_11} and \ref{sec:Ch_12}, respectively, and identifies the Minnaert resonances in Section \ref{sec:Ch_2}.

\subsection{Scattering resonances in the physical and rescaled domains}

\subsubsection{Physical-domain formulation} \label{sec:Ch_11}
Set
\begin{align*}
c_-:=\sqrt{\frac{k_-}{\rho_-}}.
\end{align*}
Let $\nu_D$ denote the unit normal on $\partial D_\eps$ pointing from
$D_\eps$ into $\Omega_+\setminus\overline{D_\eps}$.
For traces on $\Gamma$, we use the convention
\[
[v]_\Gamma:=v_+-v_-,
\]
and
\[
\left[
\frac1{\rho^0}\partial_{\nu_\Gamma}v
\right]_\Gamma
:=
\frac1{\rho_+}\partial_{\nu_\Gamma}v_+
-
\frac1{\rho_-}\partial_{\nu_\Gamma}v_-.
\]

Since \(H_{\rho_\varepsilon,k_\varepsilon}\) is a black-box Hamiltonian, by Theorem~4.9 in \cite{DM}, those nonzero resonances admit following characterization of resonances in the physical domain.

\begin{proposition}\label{Pr:1}
A complex number $z \in \mathbb C\backslash \{0\}$ is a resonance of the acoustic
transmission problem if there exists a nonzero field
\[
u_\eps
\in
H^1_{\loc}
\bigl(
\R^3\setminus(\Gamma\cup\partial D_\eps)
\bigr)
\]
such that
\begin{align*}
\nabla\cdot
\left(
\frac1{\rho_\eps}\nabla u_\eps
\right)
+
\frac{z^2}{k_\eps}u_\eps
=0
\end{align*}
in each of the three regions
\(
D_\eps,
\;
\Omega_+\setminus\overline{D_\eps},
\) and \(\Omega_-,\)
and such that
\begin{align*}
&[u_\eps]_\Gamma=0,
\qquad
\left[
\frac1{\rho^0}
\partial_{\nu_\Gamma}u_\eps
\right]_\Gamma
=0
\qquad
\mathrm{on}\; \Gamma,\\
& u_\eps^{\mathrm{int}}
=
u_\eps^{\mathrm{ext}},\qquad
\frac1{\rho_b^\eps}
\partial_{\nu_D}u_\eps^{\mathrm{int}}
=
\frac1{\rho_+}
\partial_{\nu_D}u_\eps^{\mathrm{ext}}
\qquad
\mathrm{on}\; \partial D_\eps,\\
& u_\eps\;\mathrm{is}\; {z/c_--}\; \mathrm{outgoing}.
\end{align*}
Here, we call that $v$ is $z-\rm{outgoing}$ means that there exists $g\in L_{\mathrm{comp}}^2(\R^3)$ and $r>0$ such that 
\begin{align*}
v(x) =  \int_{\R^3}\frac{e^{iz|x-y|}}{4\pi|x-y|} g(y) dy \qquad x\in \R^3 \backslash \overline{B_r},
\end{align*}
where $L^2_{\mathrm{comp}}(\R^3)$ is defined by
\begin{align*}
L^2_{\mathrm{comp}}(\R^3):=\{u\in L^2(\R^3): \exists\; r>0, |u(x)| = 0\; \mathrm{for}\; |x|>r\}.
\end{align*}
\end{proposition}

This proposition yields an alternative characterization of resonances in terms of the existence of a nontrivial solution to the following boundary value problem. Before doing so, we introduce some additional notation.
Define the exterior domain
\[
E_\eps:=\R^3\setminus\overline{D_\eps}.
\]
Let $\mathcal U_\eps\subset\C$ be an open set on which the outgoing
exterior Dirichlet transmission problem below is uniquely solvable.
For $z\in\mathcal U_\eps$ and
$f\in H^{1/2}(\partial D_\eps)$, let $U_\eps^f(z)$ denote the unique
outgoing solution of
\[
\nabla\cdot
\left(
\frac1{\rho^0}\nabla U_\eps^f(z)
\right)
+
\frac{z^2}{k^0}U_\eps^f(z)
=0
\qquad
\mathrm{in}\; E_\eps\setminus\Gamma,
\]
subject to
\begin{align*}
&[U_\eps^f(z)]_\Gamma=0, \qquad
\left[
\frac1{\rho^0}
\partial_{\nu_\Gamma}U_\eps^f(z)
\right]_\Gamma
=0,\\
&U_\eps^f(z)=f
\qquad
\mathrm{on }\;\partial D_\eps,
\end{align*}
and the ${z/c_--}$ outgoing condition at infinity. We further define
\begin{align}\label{eq:D_to_N}
\mathcal N_\eps(z):
H^{1/2}(\partial D_\eps)
\longrightarrow
H^{-1/2}(\partial D_\eps),\qquad \mathcal N_\eps(z)f
:=
\left.
\frac1{\rho_+}
\partial_{\nu_D}U_\eps^f(z)
\right|_{\partial D_\eps}.
\end{align}
We further introduce the sesquilinear form
\begin{align*}
t_\eps(z;u,v)
:={}&
\int_{D_\eps}
\frac1{\rho_b^\eps}
\nabla u\cdot\nabla\overline v
\,\dd x -
z^2
\int_{D_\eps}
\frac1{k_b^\eps}
u\overline v
\,\dd x\\ 
&-
\left\langle
\mathcal N_\eps(z)\gamma u,
\gamma v
\right\rangle_{
H^{-1/2}(\partial D_\eps),
H^{1/2}(\partial D_\eps)
}, \quad u,v\in H^1(D_\eps),
\end{align*}
and let
\[
\mathcal T_\eps(z):
H^1(D_\eps)
\longrightarrow
H^1(D_\eps)
\]
be defined by
\[
\langle T_\eps(z)u,v\rangle_{H^1(D_\eps)}
:=
t_\eps(z;u,v).
\]
A nontrivial solution \(u\) of $\mathcal T_\eps (z) u = 0 $ for some $z\in\mathcal U_\eps$ naturally satisfies the below boundary value problem:
\begin{align*}
& \nabla\cdot
\left(
\frac1{\rho_b^\eps}\nabla u
\right)
+
\frac{z^2}{k_b^\eps}u
=0
\qquad
\text{in }D_\eps,\\
&
\frac1{\rho_b^\eps}
\partial_{\nu_D}u
-
\mathcal N_\eps(z)\gamma u
=0
\qquad
\text{on}\;\partial D_\eps.
\end{align*}
Following the argument of Proposition~4.3 in \cite{LiSinielastic}, together with
Proposition \ref{Pr:1} and the above observation, we obtain the
following equivalent characterization of the resonances.

\begin{proposition}\label{pro:equi}
Fixed \(\varepsilon > 0\) Let $z \in \mathcal U_\eps$. 
$z \in \C \backslash \{0\}$ is a resonance of the full acoustic transmission problem, if and only if there exists $u \neq 0\in H^1(D_\eps)$ such that
\[
\mathcal T_\eps(z)u=0.
\]
\end{proposition}

We note that the open set \(\mathcal U_\varepsilon\), on which the exterior Dirichlet-to-Neumann map \(\mathcal N_\varepsilon\) defined in \eqref{eq:D_to_N} is well defined, depend on \(\varepsilon\).
Since the analysis below compares
the problems for different values of \(\varepsilon\), we need a fixed
spectral region on which these inverses exist uniformly. Therefore, in the remainder of this section, we aim to show that, for
every compact set \(K\subset\mathcal O\), the operator
\(\mathcal N_\varepsilon(z)\) is well defined on \(K\) for all
sufficiently small \(\varepsilon\). 

Let $R > 2$ be fixed so that the closure of $\Omega$ is contained in $B_{R-2}$. Here, $B_r:=\{x \in \R^3: |x| < r\}$ for $r>0$. On the space  
\[
V_\eps:=\{v\in H^1(B_R\setminus\overline{D_\eps}):\gamma_{\partial D_\eps}v=0\},
\] 
we define the sesquilinear form 
\begin{align} 
\label{eq:truncated-form}
\mathfrak b_{\eps,z}(u,v)
:={}&
\int_{B_R\setminus\overline{D_\eps}}
\bigl((\rho^0)^{-1}\nabla u\cdot\nabla\overline v-z^2 (k^{(0)})^{-1}\,u\overline v\bigr)(x)\,\dd x
\notag\\
&-
\frac{1}{\rho_-}\left\langle
\Lambda_R^{\rm out}(z/c_-)\gamma_Ru,
\gamma_Rv
\right\rangle_{H^{-1/2}(\partial B_R),H^{1/2}(\partial B_R)}, \quad\mathrm{for}\; u,v\in V_\eps.
\end{align}
Here, $\Lambda_R^{\rm out}(z)$ denote the exact exterior outgoing
Dirichlet-to-Neumann map on \(\partial B_R\) for
\[
(\Delta+z^2)u=0
\qquad\text{in }\mathbb R^3\setminus\overline{B_R}.
\] 
It is known that \(z\mapsto\Lambda_R^{\rm out}(z)\) is a meromorphic family in
\(
\mathcal L\bigl(
H^{1/2}(\partial B_R),
H^{-1/2}(\partial B_R)
\bigr).
\)
In~\eqref{eq:truncated-form}, we restrict \(z\) to the pole-free open set on which \(\Lambda_R^{\rm out}(z)\) is holomorphic.
Clearly, all functions in $V_\eps$ are always extended by zero through $D_\eps$.
Let
\[
\mathfrak A_\eps(z):V_\eps\longrightarrow V_\eps
\]
be the operator induced by the sesquilinear form \eqref{eq:truncated-form}. For \(\eps = 0\), we set \(V_\eps:= H^1(B_R)\).

\begin{remark}
For \(R>R_*\), the spherical-harmonic representation of the outgoing
Dirichlet-to-Neumann map \(\Lambda_R^{\mathrm{out}}(z)\) gives its
pole set as
\[
    \mathcal P_R=R^{-1}\Sigma,
    \qquad
    \Sigma:=
    \bigcup_{\ell\ge0}
    \left\{
        \zeta\in\mathbb C\setminus\{0\}:
        h_\ell^{(1)}(\zeta)=0
    \right\}.
\]
The map is understood by holomorphic extension at \(z=0\).
Each \(h_\ell^{(1)}\) has finitely many nonzero zeros
\cite{DLMF}, so \(\Sigma\) is countable.
Hence, given \(R_1>R_*\), we may choose \(R_2>R_1\) such that
\(R_2/R_1\) avoids all ratios of elements of \(\Sigma\), yielding
\[
    \mathcal P_{R_1}\cap\mathcal P_{R_2}=\varnothing.
\]

For \(z\notin\mathcal P_R\), the exact truncated formulation
is equivalent to the background transmission problem.
Hence \(\mathfrak A_0(z): V_0\longrightarrow V_0\) is invertible if and only if
\(z\in\mathcal O\). If a point \(z\in\mathcal O\) belongs
to \(\mathcal P_R\), another truncation radius can be chosen
for which the exterior Dirichlet-to-Neumann map is holomorphic
near \(z\), and the corresponding truncated operator is
invertible. Thus the auxiliary poles impose no additional
exclusions on \(\mathcal O\), which is independent of the
truncation radius. In each local argument below, we fix
an admissible radius and denote it simply by \(R\).
\end{remark}

\begin{proposition}
\label{le:uniform-inverse}
Let $K\subset\mathcal O$ be compact. Then there exist
$\eps_K>0$ such that
\[
\sup_{0<\eps<\eps_K}\sup_{z\in K}
\|\mathfrak A_\eps(z)^{-1}\|_{\mathcal L(V_\eps,V_\eps)}
\le C_K.
\]
Here, $C_K$ is a positive constant independent of $\eps_K$.
\end{proposition}

The proof of Proposition \ref{le:uniform-inverse} will be given in Appendix \ref{app:uniform-stability}, which is based on a contradiction and compactness procedure: a sequence violating the estimate would produce a nontrivial solution
of the limiting homogeneous background problem, contradicting the
definition of \(\mathcal O\).

Under the assumptions of Proposition \ref{le:uniform-inverse}, it follows directly that \(\mathcal N_\varepsilon(z)\), specified by \eqref{eq:D_to_N}, is well defined for every \(z\in K\) and \(\eps \in (0,\eps_ K)\).

\subsubsection{Rescaling and resonance equivalence} \label{sec:Ch_12}

In this section, we pull the bounded-domain
problem back to the fixed reference domain $B$ and study its associated equivalent resonance characterization.
We define the reference-to-physical transport operators
\begin{align*}
&\left(\mathcal P_\eps U\right)(x):= U\left( \frac{x-x_\eps}{\eps}\right), \quad x \in \R^3, \;\; U \in H^1(B),\\
&\left(\mathcal P_\eps^\partial V\right)(x)
:= V\left( \frac{x-x_\eps}{\eps}\right), \quad x \in \partial B, \;\; V \in L^2(\partial B).
\end{align*}
Therefore,
\[
\gamma_{\partial D_\eps}\mathcal P_\eps U
=
\mathcal P_\eps^\partial\gamma_{\partial B}U.
\]
For \(f,g\in H^{1/2}(\partial B)\), define the rescaled exterior
Dirichlet-to-Neumann form by
\begin{align}\label{eq:scaled_2}
\mathfrak q_\eps(z;f,g)
:=
-\frac1\eps
\left\langle
\mathcal N_\eps(z)\mathcal P_\eps^\partial f,
\mathcal P_\eps^\partial g
\right\rangle_{
H^{-1/2}(\partial D_\eps),
H^{1/2}(\partial D_\eps)
}.
\end{align}
Furthermore, let
\begin{align}\label{eq:scaled_1}
\widehat t_\eps(z;U,V)
:=
\frac1\eps
t_\eps
\bigl(
z;\mathcal P_\eps U,\mathcal P_\eps V
\bigr),\quad \mathrm{for}\; U,V\in H^1(B),
\end{align}
and let
\[
\widehat {\mathcal T}_\eps(z):
H^1(B)\longrightarrow H^1(B)
\]
be defined by
\[
\langle\widehat {\mathcal T}_\eps(z)U,V\rangle_{H^1(B)}
:=
\widehat t_\eps(z;U,V).
\]
Since $\mathcal P_\eps$ is invertible for each $\eps >0$, we have 
\begin{align}\label{eq:10}
\ker \mathcal T_\eps(z)\neq\{0\}
\quad\Longleftrightarrow\quad
\ker\widehat {\mathcal T}_\eps(z)\neq\{0\}.
\end{align}
By \eqref{eq:10}, it suffices to study the kernel of \(\widehat {\mathcal T}_\eps\) for the problem posed on the fixed reference domain \(B\). The following lemma provides an explicit representation of the associated sesquiliner form \(\widehat t_\eps\).

\begin{lemma}\label{le:1}
For every \(U,V\in H^1(B)\), we have 
\begin{align}\label{eq:variational_1}
\widehat t_\eps(z;U,V)
={}&
\frac1{\bar\rho_b\eps^2}
\int_B\nabla U({\widetilde x})\cdot\nabla\overline V({\widetilde x})\,\dd {\widetilde x}
-
\frac{z^2}{\bar k_b}
\int_B U({\widetilde x})\overline V({\widetilde x})\,\dd {\widetilde x} +
\mathfrak q_\eps(z;\gamma U,\gamma V).
\end{align}
Here, \(\widehat t_\eps\) and \(\mathfrak q_\eps\) are defined by \eqref{eq:scaled_1} and \eqref{eq:scaled_2}, respectively.
\end{lemma}

\begin{proof}
For \(U,V\in H^1(B)\), set
\[
u_\eps:=\mathcal P_\eps U
\quad \mathrm{and} \quad 
v_\eps:=\mathcal P_\eps V.
\]
We also denote by \(u_\eps\) its extension, obtained as the outgoing solution with the trace $\gamma_{\partial D_\eps} u_\eps$.
Introducing a change of variables
\(
x= x_\eps + \eps {\widetilde x},\)
by the chain rule, we have
\begin{align*}
\nabla_x u_\eps(x_\eps + \eps {\widetilde x})
=
\frac1\eps\nabla_{\widetilde x}U({\widetilde x}),
\qquad
\nabla_x v_\eps(x_\eps + \eps {\widetilde x})
=
\frac1\eps\nabla_{\widetilde x}V({\widetilde x}).
\end{align*}
Hence, we have 
\begin{align*}
\frac1\eps
\int_{D_\eps}
\frac1{\rho_b^\eps}
\nabla u_\eps(x)\cdot
\nabla\overline{v_\eps}(x)\,\dd x
&=
\frac1\eps
\int_B
\frac1{\rho_b^\eps}
\left(
\frac1\eps\nabla U ({\widetilde x})
\right)
\cdot
\left(
\frac1\eps\nabla\overline V ({\widetilde x})
\right)
\eps^3\,\dd {\widetilde x}\\
&=
\frac1{\rho_b^\eps}
\int_B
\nabla U({\widetilde x})\cdot\nabla\overline V({\widetilde x})\,\dd {\widetilde x}.
\end{align*}
Using the relation
\(
\rho_b^\eps=\bar\rho_b\eps^2,
\)
we obtain
\begin{align}\label{eq:17}
\frac1\eps
\int_{D_\eps}
\frac1{\rho_b^\eps}
\nabla u_\eps(x) \cdot
\nabla\overline{v_\eps}(x)\,\dd x
=
\frac1{\bar\rho_b\eps^2}
\int_B
\nabla U({\widetilde x})\cdot\nabla\overline V({\widetilde x})\,\dd {\widetilde x}.
\end{align}
Similarly,
\[
\begin{aligned}
-\frac{z^2}{\eps}
\int_{D_\eps}
\frac1{k_b^\eps}
u_\eps(x)\overline{v_\eps}(x)\,\dd x
&=
-\frac{z^2}{\eps}
\int_B
\frac1{k_b^\eps}
U({\widetilde x})\overline V({\widetilde x})\,\eps^3\,\dd {\widetilde x}
=
-\frac{z^2\eps^2}{k_b^\eps}
\int_B
U({\widetilde x})\overline V({\widetilde x})\,\dd {\widetilde x}.
\end{aligned}
\]
Since
$ k_b^\eps=\bar k_b\eps^2, $
this becomes
\begin{align}\label{eq:18}
-\frac{z^2}{\eps}
\int_{D_\eps}
\frac1{k_b^\eps}
u_\eps(x)\overline{v_\eps}(x)\,\dd x
=
-\frac{z^2}{\bar k_b}
\int_B
U({\widetilde x})\overline V({\widetilde x})\,\dd {\widetilde x}.
\end{align}
Moreover, since 
\[
\begin{aligned}
\bigl(
\gamma_{\partial D_\eps}u_\eps
\bigr)(x)
&= 
\bigl(
\mathcal P_\eps^\partial\gamma_{\partial B}U
\bigr)(x),
\end{aligned}
\]
it can be seen that
\begin{align*}
&\mathfrak q_\eps(z;\gamma U,\gamma V)=-\frac1\eps
\left\langle
\mathcal N_\eps(z)
\gamma_{\partial D_\eps}u_\eps,
\gamma_{\partial D_\eps}v_\eps
\right\rangle_{
H^{-1/2}(\partial D_\eps),
H^{1/2}(\partial D_\eps)
}.
\end{align*}
Combining \eqref{eq:scaled_1}, \eqref{eq:17} and \eqref{eq:18} gives the identity \eqref{eq:variational_1}.
\end{proof}

\subsection{Identification of the Minnaert resonance} \label{sec:Ch_2}

In this section, we identify the Minnaert resonance by studying the moderate characteristic values of the operator family
\(\widehat {\mathcal T}_\eps(z)\). To do so, we decompose $H^1(B)$ into the constant mode and its
mean-zero complement.
Define
\[
m(u):=\frac{1}{|B|}\int_B u(\widetilde x)\,\dd {\widetilde x}, \qquad u\in H^1(B),
\]
and
\[
H_0
:=
\left\{
w\in H^1(B):
\int_B w(\widetilde x)\,\dd {\widetilde x}=0
\right\}.
\]
Then
\[
H^1(B)
=
\Span\{\one\}\oplus H_0,
\]
and every $u\in H^1(B)$ has the unique decomposition
\begin{align} \label{eq:dec}
u = m(u)\one + w,
\qquad
w\in H_0.
\end{align}
On $H_0$, we use the norm
\begin{align}\label{eq:5}
\|w\|_{H_0}:=\|\nabla w\|_{L^2(B)},
\end{align}
which is equivalent to the $H^1(B)$ norm by the Poincar\'e
inequality.

With the aid of the above decomposition \eqref{eq:dec}, it can be seen that the equation 
\[
\widehat T_\eps(z)u=0
\]
is equivalent to
\begin{align}
&\widehat t_\eps(z; m(u)\one + w,\one)=0, \label{eq:6}\\
\mathrm{and}\;
&\widehat t_\eps(z; m(u)\one + w,v)=0
\qquad
\mathrm{for}\;\mathrm{every}\; v\in H_0. \label{eq:7}
\end{align}
For $z\in \mathcal O$, we define
\begin{align}
&\mathfrak C_\eps(z): H_0\longrightarrow H_0,\notag \\
\mathrm{and}\; &\mathfrak L_\eps(z): H_0\longrightarrow\C\notag
\end{align}
by
\[
\left\langle
\mathfrak C_\eps(z)w,v 
\right\rangle_{H_0}
:=
\widehat t_\eps(z;w,v),
\qquad
w,\;v\in H_0,
\]
and 
\begin{align}\label{eq:32}
\mathfrak L_\eps(z)w
:=
\widehat t_\eps(z;w,\one),\quad w \in H_0,
\end{align}
respectively. Furthermore, we define
\begin{align}\label{eq:22}
\mathfrak A_\eps(z)
:=
\widehat t_\eps(z;\one,\one),
\end{align}
and denote
\(
\mathrm R_\eps(z)\in H_0
\)
by
\begin{align}\label{eq:33}
\left\langle
\mathrm R_\eps(z),v
\right\rangle_{H_0}
:=
\widehat t_\eps(z;\one,v),
\qquad v\in H_0,
\end{align}
Equations \eqref{eq:6}--\eqref{eq:7} are equivalent to 
\begin{align}
&m(u)\,\mathfrak A_\eps(z)+ \mathfrak L_\eps(z)w=0, \label{eq:8}\\
\mathrm{and}\;& \mathfrak C_\eps(z)w + m(u)\,\mathfrak R_\eps(z)=0. \label{eq:9}
\end{align}

We now investigate the invertibility of \(\mathfrak C_\eps\), for which we first establish a uniform bound for the scaled Dirichlet-to-Neumann form defined in \eqref{eq:scaled_2}.

\begin{lemma}
\label{lem:main-uniform-q}
Let $K\subset \mathcal O$ be compact. Then there exists $\eps_K>0$
such that
\[
|\mathfrak q_\eps(z;f,g)|
\le C_K
\|f\|_{H^{1/2}(\partial B)}
\|g\|_{H^{1/2}(\partial B)}
\]
for every $z\in K$, $0<\eps<\eps_K$, and
$f,g\in H^{1/2}(\partial B)$. Here, $C_K$ is a positive constant independent of $\eps_K$ and the $\operatorname{dist}(D_\eps,\Gamma)$.
\end{lemma}

\begin{proof}
We choose a bounded $C^2-$ smooth domain $B^\sharp$ with
$\overline B\subset B^\sharp$. We first claim that there exists a bounded right inverse of the trace,
\[
L:H^{1/2}(\partial B)
\longrightarrow H^1(B^\sharp\setminus\overline B),
\]
such that
\[
\gamma_{\partial B}Lf=f,
\qquad
\gamma_{\partial B^\sharp}Lf=0,
\qquad
\|Lf\|_{H^1}\le C_{B,B^{\sharp}}\|f\|_{H^{1/2}(\partial B)}.
\] 
For each $f \in H^{1/2}(\partial B)$, we define
\[
[L_\eps] f(x)
:=\begin{cases}
[Lf]\!\left(\frac{x-x_\eps}{\eps}\right), \quad\; & \mathrm{in}\; \left\{y\in\R^3 : y = x_\eps +\eps(B^\sharp\setminus\overline B)\right\}.\\
0, & \mathrm{else}
\end{cases}
\]
Therefore, it can be seen that 
\begin{align*}
&L_\eps:H^{1/2}(\partial B)\longrightarrow H^1(B_R \backslash \overline D_\eps)
\end{align*}
{and} that 
\begin{align*}
\gamma_{\partial D_\eps}L_\eps f= f\left(\frac{x-x_\eps}{\eps}\right),
\qquad
\gamma_RL_\eps f=0.
\end{align*}
Since
\[
\|L_\eps f\|_{L^2}^2
=\eps^3\|Lf\|_{L^2}^2,
\qquad
\|\nabla L_\eps f\|_{L^2}^2
=\eps\|\nabla Lf\|_{L^2}^2,
\]
we readily obtain that 
\begin{align} \label{eq:11}
\|L_\eps f\|_{H^1(B_R\backslash \overline{D_\eps})}
\le C\eps^{1/2}\|f\|_{H^{1/2}(\partial B)}.
\end{align}
The constant is independent of $\operatorname{dist}(D_\eps,\Gamma)$.

On the other hand, let
$U_{\eps,z}^f$ be the outgoing exterior solution with
$U_{\eps,z}^f = f\left({(x-x_\eps)}/{\eps}\right)$ on $\partial D_\eps$, then there exists $v_{\eps,z}\in V_\eps$ such that 
\[
U_{\eps,z}^f = L_\eps f + v_{\eps,z},
\]
and that 
\begin{align*}
 \mathfrak b_{\eps,z}
\bigl(v_{\eps,z}, h \bigr) = \int_{B_R\setminus\overline{D_\eps}}-
\bigl((\rho^0)^{-1}\nabla (L_\eps f)\cdot\nabla\overline h - z^2 (k^{(0)})^{-1}\,(L_\eps f)\overline h\bigr)(x)\,\dd x, \quad h \in V_\eps.
\end{align*}
Here, the sesqulilnear form \(\mathfrak b_{\eps,z}\) is specified by \eqref{eq:truncated-form}. Therefore, using \eqref{eq:11}, we have 
\[
|\langle \mathfrak A_\eps(z)v_{\eps,z}, h \rangle_{V_\eps}| \le C_K \eps^{1/2} \|h\|_{V_\eps} \|f\|_{H^{1/2}(\partial B)}.
\]
This implies
\begin{align*}
   \|\mathfrak A_\eps(z)v_{\eps,z}\|_{{H^1(B_R\backslash \overline{D_\eps})}} \le C_K \eps^{1/2} \|f\|_{H^{1/2}(\partial B)}
\end{align*}
In conjunction with Proposition \ref{le:uniform-inverse}, we arrive at 
\[
\|v_{\eps,z}\|_{H^1(B_R\backslash \overline{D_\eps})}
\le C_K\eps^{1/2}
\|f\|_{H^{1/2}(\partial B)}.
\]
Consequently,
\[
\|U_{\eps,z}^f\|_{H^1(B_R\backslash \overline{D_\eps})}
\le C_K\eps^{1/2}
\|f\|_{H^{1/2}(\partial B)}.
\]
Since $L_\eps g$ vanishes on $\partial B_R$, Green's formula gives
\[
\eps \mathfrak q_\eps(z;f,g)
= \int_{B_R\setminus\overline{D_\eps}}
\bigl((\rho^0)^{-1}\nabla (U_{\eps,z}^f)\cdot\nabla\overline {L_\eps g} - z^2 (k^{(0)})^{-1}\,(U_{\eps,z}^f)\overline {L_\eps g}\bigr)(x)\,\dd x.
\]
Therefore
\begin{align*}
\eps|\mathfrak q_\eps(z;f,g)|
&\le C_K
\|U_{\eps,z}^f\|_{H^1(B_R\backslash \overline{D_\eps})}
\|L_\eps g\|_{H^1(B_R\backslash \overline{D_\eps})}\\
&\le C_K\eps
\|f\|_{H^{1/2}(\partial B)}
\|g\|_{H^{1/2}(\partial B)},
\end{align*}
which proves the assertion.

\end{proof}

Building upon Lemma \ref{lem:main-uniform-q}, we are now ready to show the invertibility of \(\mathfrak C_\eps\) for sufficiently small \(\eps\) in the following lemma.

\begin{lemma}
\label{lem:mean-zero-invertibility}
For every compact set
$K\subset  \mathcal O$, there exists $\eps_K>0$ such that,
for every
\( z\in K,\; 0<\eps<\eps_K,
\)
the operator
\(
\mathfrak C_\eps(z):H_0\longrightarrow H_0
\)
is invertible and satisfies
\[
\left\|
\mathfrak C_\eps(z)^{-1}
\right\|_{\mathcal L(H_0)}
\le
C_K\eps^2.
\]
Here, $C_K$ is a positive constant independent of $\eps_K$ and the $\operatorname{dist}(D_\eps,\Gamma)$.
\end{lemma}

\begin{proof}
It follows from Lemma \ref{lem:main-uniform-q} and \eqref{eq:5} that there exist $\eps_K>0$ such that
\[ 
\left|
-\frac{z^2}{\bar k_b}
\int_B |w(\widetilde x)|^2\,\dd {\widetilde x}
\right|
\le
C_K\|\nabla w\|_{L^2(B)}^2,\qquad \mathrm{uniformly\; for}\; z\in K,\;\; \mathrm{and}\; \eps\in(0,\eps_K).
\]
This, together with the definition of the operator $\mathfrak C_\eps(z)$ gives
\[
\operatorname{Re}
\left\langle
\mathfrak C_\eps(z)w,w
\right\rangle_{H_0}
\ge
\left(
\frac{1}{\bar\rho_b\eps^2}
-
C_K
\right)
\|\nabla w\|_{L^2(B)}^2.
\]
Therefore, for sufficiently small $\eps$,
\[
\operatorname{Re}
\left\langle
\mathfrak C_\eps(z)w,w
\right\rangle_{H_0}
\ge
\frac{1}{2\bar\rho_b\eps^2}
\|w\|_{H_0}^2.
\]
Using the complex Lax--Milgram lemma, we readily obtain the invertibility of
$\mathfrak C_\eps(z)$, together with
\[
\left\|
\mathfrak C_\eps(z)^{-1}
\right\|_{\mathcal L(H_0)}
\le
2\bar\rho_b\eps^2, \qquad \mathrm{uniformly\; for}\; z\in K,\;\; \mathrm{and}\; \eps\in(0,\eps_K).
\]
The proof of this lemma is thus completed.
\end{proof}

We conclude this section with an alternative characterization of resonances in a given compact region.

\begin{lemma}
\label{le:exact-schur}
For every compact set
$K\subset  \mathcal O$, there exists $\eps_K>0$ such that,
for every
\( z\in K,\; 0<\eps<\eps_K,
\) the following arguments hold true. 

\begin{enumerate}[(a)]
\item \label{a1}
$z \in K \{0\}$ is a resonance of the full acoustic transmission problem if and only if 
\begin{align}
F_\eps(z)=0. \notag
\end{align}
Here, 
\begin{align}
F_\eps(z)
:=
\mathfrak A_\eps(z)
-
\mathfrak L_\eps(z)
\mathfrak C_\eps(z)^{-1}
\mathfrak R_\eps(z). \notag
\end{align}
\item \label{a2} We have
\[
\sup_{z\in K}
\left|
F_\eps(z)- \mathfrak A_\eps(z)
\right|
\le C_K\eps^2.
\]
Here, $C_K$ is a positive constant independent of $\eps_K$ and $\operatorname{dist}(D_\eps,\Gamma)$.
\end{enumerate}
\end{lemma}

\begin{proof}
\eqref{a1} With the aid of \eqref{eq:10}, \eqref{eq:9} and Lemma \ref{lem:mean-zero-invertibility}, for $u \in \Ker\widehat T_\eps(z)$ we obtain that it admits the decomposition \eqref{eq:dec}, 
\[
w =- m(u)\,\mathfrak C_\eps(z)^{-1}\mathfrak R_\eps(z).
\]
Substituting the above identity into \eqref{eq:8} gives
\[
m(u)F_\eps(z)=0.
\]
If $m(u)=0$, then
\[
\mathfrak C_\eps(z)w=0,
\]
and hence $W=0$. Thus a nonzero kernel element necessarily satisfies
$m(u)\neq0$. Consequently, we readily obtain
\[
\Ker\widehat T_\eps(z)\neq\{0\}
\quad\Longleftrightarrow\quad
F_\eps(z)=0.
\]
This, together with \eqref{eq:10} and Proposition \ref{pro:equi} yields the statement.
 
\eqref{a2}
With the aid of \eqref{eq:32} and \eqref{eq:33}, it follows from Lemma \ref{lem:main-uniform-q} that there exist $\eps_K>0$ and $C_K>0$ such that,
for every
\( z\in K\) and \(0<\eps<\eps_K,
\)
\[
\|\mathfrak R _\eps(z)\|_{H_0}
+
\|\mathfrak L_\eps(z)\|_{\mathcal L(H_0,\C)}
\le C_K.
\]
This, together with Lemma \ref{lem:mean-zero-invertibility} yields that for every
\( z\in K,\; 0<\eps<\eps_K,
\)
\begin{align*}
|F_\eps(z) - \mathfrak A_\eps(z)|
&=
\left|
\mathfrak L_\eps(z)
\mathfrak C_\eps(z)^{-1}
\mathfrak R_\eps(z)
\right|
\le C_K\eps^2.
\end{align*}
The proof of this lemma is thus completed. 
\end{proof}

\section{Proof of Theorem \ref{th:main}} \label{sec:p}

This section is devoted to proving Theorem \ref{th:main}. To this end, we prepare several lemmas.

\begin{lemma}
\label{le:main-D2}
Let $\eps > 0$. The following arguments hold true.
\begin{enumerate}[(a)]
\item \label{d1} Let $P_\eps$ denote the static capacitary potential of $D_\eps$ in
the background, which is a solution of 
\begin{align*}
&\nabla\cdot
\left(
\frac1{\rho^0}\nabla P_\eps
\right)
=0\; \qquad\qquad\qquad\quad\; \mathrm{in}\; \R^3 \backslash \overline{ D_\eps},\\
&[P_\eps]_\Gamma=0,
\qquad
\left[
\frac1{\rho^0}
\partial_{\nu_\Gamma}P_\eps
\right]_\Gamma
=0
\qquad
\mathrm{on}\; \Gamma,\\
& P_\eps^{\mathrm{int}}
 = P_\eps^{\mathrm{ext}}
\qquad\qquad\qquad\qquad\qquad\;\; \mathrm{on}\; \partial D_\eps,\\
& P_\eps = 1 \qquad\qquad\qquad\qquad\qquad\qquad\;\;\; \mathrm{in}\; D_\eps,\;\\
&\lim_{|x|\rightarrow +\infty} P_\eps(x) = 0.
\end{align*}
We have
\begin{align}\label{eq:41}
\int_{B_R\setminus  \overline{D_\eps}}|P_\eps(x)|^2\,\dd x 
\le C\eps^2, \qquad \mathrm{as}\; \eps\to 0.
\end{align}
Moreover, if \(A\) is a fixed annulus such that $ B_{R-2} \subset A \subset B_R$, then
\begin{align}\label{eq:42}
\|P_\eps\|_{H^1(A)}\le C\eps, \qquad \mathrm{as}\; \eps\to 0.
\end{align}

\item \label{d2} Let $K\in \mathcal O$. We have 
\begin{align} \label{eq:55}
\sup_{z\in K}
|\mathfrak q_\eps(z;\one,\one)- \mathfrak q_\eps(0;\one,\one)|
 \le C\eps \qquad \mathrm{as}\; \eps\to 0.
\end{align}
\end{enumerate}

Here, $C$ is a positive constant independent of \(\eps\) and
$\dist(D_\eps,\Gamma)/\eps$.
\end{lemma}

\begin{proof}
Let $\eps >0$ be sufficiently small throughout the proof.

\eqref{d1}
It is known that 
\[
\icap^{\mathrm{phys}}(D_\varepsilon)
=
\int_{\R^3\setminus \overline{D_\eps}} \frac{1}{\rho^{(0)}(x)}|\nabla P_\eps(x)|^2\,\dd x.
\]
We now claim that there exists $C>0$, independent of the scaled gap and the small parameter $\eps$, such that
\begin{align} \label{eq:31}
\icap^{\mathrm{phys}}(D_\eps)\le C\eps, \quad \mathrm{as}\;\; \eps \rightarrow 0.
\end{align}
Choose $\chi\in C_c^\infty(\R^3)$ equal to one on a fixed ball
containing $D$, with $\nabla\chi$ supported in a
fixed annulus $A$. By a straightforward calculation, we have
\begin{align}
\int_{\R^3\setminus  \overline{D_\eps}}\frac{1}{\rho^{(0)}(x)}|\nabla \chi\left((x-x_\eps)/\eps\right)|^2\,\dd x = \eps\int_{\R^3\backslash \overline B} \frac{1}{\rho^{(0)}(y)}|\nabla \chi \left(y\right)|^2\,\dd y \le C_\chi \eps. \notag
\end{align}
This, together with \eqref{eq:phy} shows that \eqref{eq:31} holds. 

Let $G_{\rho^0}(x,y)$ be the positive whole-space Green function of
$-\nabla\cdot(1/\rho^0\nabla)$.
Using the Green representation formula, we have the representation of $P_\eps$
\begin{align}\label{eq:38}
P_\varepsilon(x)
=
-\int_{\partial D_\varepsilon}
G_{\rho^{(0)}}(x,y)\,
(\rho^0)^{-1}\partial_{\nu_D}P_\varepsilon(y)\,dS(y),
\end{align}
By the maximum principle,
\begin{align} \label{eq:44}
0<P_\varepsilon<1
\qquad\text{in }\mathbb R^3\setminus\overline{D_\varepsilon}.
\end{align}
Since \(\partial D_\varepsilon\) is smooth, Hopf's
lemma (see, e.g.,\cite[Theorem 2.5]{HL}) yields
\begin{align}\label{eq:39}
\partial_{\nu_D}P_\varepsilon<0
\qquad\text{on }\partial D_\varepsilon.
\end{align}
Furthermore, 
\begin{align*}
\int_{B_R\setminus D_\eps}|P_\eps(x)|^2\,\dd x = \int_{B_{c_B\eps}(x_\eps)}|P_\eps(x)|^2\,\dd x + \int_{B_R \backslash \overline{B_{c_B\eps}(x_\eps)}}|P_\eps(x)|^2\,\dd x =: I^{(1)}_\eps +  I^{(2)}_\eps.
\end{align*}
Here, \(c_B>0\) is chosen such that \(c_B\) is equal to twice the diameter of \(B\). Clearly,
\(D_\eps\subset B_{c_B\eps}(x_\eps)\). With the aid of \eqref{eq:44}, we have 
\begin{align}\label{eq:34}
    I^{(1)}_\eps \le C\eps^3.
\end{align}
For the estimate \(I^{(2)}_\eps\), we note that 
\begin{align}\label{eq:green}
0<G_{\rho^0}(x,y)\le \frac{C}{|x-y|}, \qquad \mathrm{for}\; y\in D_\eps\; \mathrm{and}\; x\in B_R \backslash  \overline{B_{c_B\eps}(x_\eps)}.
\end{align}
By the choice of \(c_B\), we have 
\[
|x-y|\ge|x-x_\eps| - |y-x_\eps| \ge \frac12|x-x_\eps|
\qquad \mathrm{for}\; y\in D_\eps\; \mathrm{and}\; x\in B_R \backslash \overline{B_{c_B\eps}(x_\eps)}.
\]
Combining this with \eqref{eq:31}, \eqref{eq:38}, \eqref{eq:39} and \eqref{eq:green} gives
\begin{align} \label{eq:43}
P_\eps(x)
\le
\frac{C}{|x-x_\eps|}
\icap_\Gamma(D_\eps)
\le C\frac{\eps}{|x-x_\eps|}.
\end{align}
Thus, we have
\[
I^{(2)}_\eps \le C\int_{c_B\eps}^{C_M}
\left(\frac{\eps}{r}\right)^2r^2\,\dd r
\le C\eps^2.
\]
This, together with \eqref{eq:34} gives \eqref{eq:41}.

Furthermore, on a slightly larger fixed annulus $A'$ with $B_R\supset A'\supset\overline A$, the
pointwise estimate \eqref{eq:43} gives
\[
\|P_\eps\|_{L^2(A')}\le C\eps.
\]
Since $P_\eps$ is harmonic in $A'$, it follows from
Caccioppoli's inequality (see, e.g.,\cite[Lemma 1.36]{HL}) that
\[
\|\nabla P_\eps\|_{L^2(A)}
\le C\|P_\eps\|_{L^2(A')}
\le C\eps.
\]
Combining this with \eqref{eq:41} gives \eqref{eq:42}.

\eqref{d2}
Choose a cutoff real-valued function \(\widetilde \chi \in C^{\infty}_c(B_R)\)  equal to one on a fixed ball
containing $D$, with $\nabla \widetilde \chi$ supported in a
fixed annulus $A$. Define $W_{\eps}:= \widetilde \chi P_\eps$. By a straightforward calculation, we have that: for each \(v \in V_\eps\),
\begin{align*}
\int_{B_R\backslash \overline{D_\eps}}\nabla W_\eps(x) \cdot \nabla v(x) dx = \int_{B_R\backslash \overline{D_\eps}}\widetilde \chi(x)\nabla P_\eps(x) \cdot \nabla v(x) dx + \int_{B_R\backslash \overline{D_\eps}}  P_\eps(x) \nabla \widetilde \chi (x) \cdot \nabla v(x) dx. 
\end{align*}
Since $P_\eps$ is harmonic in $B_R$, it can be deduced from \eqref{eq:42} and \eqref{eq:43} that 
\begin{align} \label{eq:46}
 \left|\int_{B_R\backslash \overline{D_\eps}}\nabla W_\eps(x) \cdot \nabla v(x) dx \right| \le C\eps\|v\|_{V_\eps}.    
\end{align}
Furthermore, using \eqref{eq:42} again, we have
\begin{align}\label{eq:47}
z^2\int_{B_R}|\chi(x) P_\eps(x)|^2dx \le C_K\eps^2, \quad z \in K.
\end{align}
In conjunction with \eqref{eq:46} and \eqref{eq:47}, we arrive at 
\begin{align}\label{eq:49}
\left|\widetilde{\mathfrak b}_{\eps,z}(v, W_\eps)\right| \le C \eps \|v\|_{V_\eps}, \quad v\in V_\eps.
\end{align}
Here,
\[
\widetilde{\mathfrak b}_{\eps,z}(v, W_\eps):= \int_{B_R\setminus\overline{D_\eps}}
\bigl((\rho^0)^{-1}\nabla v \cdot\nabla\overline{W_\eps} - z^2(k^{(0)})^{-1}\, v \overline{W_\eps} \bigr)(x)\,\dd x, \qquad v\in V_\eps.
\]

Let $U_{\eps,z}^1$ be the outgoing exterior solution with
$U_{\eps,z}^1 = 1$ on $\partial D_\eps$ and set 
\begin{align}\label{eq:48}
V_{\eps,z}^1 = W_\eps - U_{\eps,z}^1\quad \mathrm{in}\; \R^3\backslash \overline {D_\eps}.
\end{align}
Clearly, 
\begin{align*}
\widetilde{\mathfrak b}_{\eps,z}(V_{\eps,z}^1, W_\eps) = \int_{B_R\setminus\overline{D_\eps}}
\bigl((\rho^0)^{-1}|\nabla W_\eps|^2 - z^2(k^{(0)})^{-1}\, |W_\eps|^2 \bigr)(x)\,\dd x \\
-\int_{B_R\setminus\overline{D_\eps}}
\bigl((\rho^0)^{-1}\nabla U^1_{\eps,z}\cdot\nabla\overline {W_\eps} - z^2(k^{(0)})^{-1}\, U^1_{\eps,z}\overline{W_\eps}\bigr)(x)\,\dd x =: J^{(1)}_{\eps,z} -  J^{(2)}_{\eps,z}.
\end{align*}
By the definition of \eqref{eq:48}, we have 
\begin{align*}
\mathfrak b_{\eps,z}(V_{\eps,z}^1,v) = \widetilde{\mathfrak b}_{\eps,z}(v,W_\eps), \quad \mathrm{for}\; v\in V_\eps.
\end{align*}
Here, the sesquilinear form \(\mathfrak b_{\eps,z}\) is specified by \eqref{eq:truncated-form}. 
Utilizing Lemma \ref{le:uniform-inverse} and applying \eqref{eq:49}, we have 
\begin{align} \label{eq:50}
  \left\|V_{\eps,z}^1\right\|_{V_\eps} \le C_K \eps, \quad  z \in K.
\end{align}
From this, utilizing \eqref{eq:49}, we have 
\begin{align} \label{eq:51}
    \left|\widetilde{\mathfrak b}_{\eps,z}(V_{\eps,z}^1, W_\eps) \right| \le C_K \eps^2, \qquad  z \in K.
\end{align}
Furthermore, with the aid of \eqref{eq:43} and \eqref{eq:47}, we have
\begin{align} \label{eq:52}
\left|J^{(1)}_{\eps,z}  - \int_{\R^3\setminus D_\eps}\frac{1}{\rho^{(0)}(x)}|\nabla P_\eps(x)|^2\,\dd x\right| \le C_K\eps^2, \qquad  z \in K.
\end{align}
Moreover, using \eqref{eq:48} and \eqref{eq:50}, we obtain
\begin{align} \label{eq:53}
\left| J^{(2)}_{\eps,z} - {\widetilde {\mathfrak b}_{\eps,z}}(U^1_{\eps,z}, U^1_{\eps,z}) + 
\frac{1}{\rho_-}\left\langle
\Lambda_R^{\rm out}(z/c_-)\gamma_RU^1_{\eps,z},
\gamma_RU^1_{\eps,z}
\right\rangle_{H^{-1/2}(\partial B_R),H^{1/2}(\partial B_R)}\right| \le C_K\eps^2, \quad  z \in K.
\end{align}

On the other hand, it follows from \eqref{eq:scaled_2} that 
\begin{align}\label{eq:45}
\eps \mathfrak q_\varepsilon(0;1,1)
=
\int_{\R^3\setminus  \overline{D_\eps}} \frac{1}{\rho^{(0)}(x)}|\nabla P_\eps(x)|^2\,\dd x,
\end{align}
and that 
\begin{align*}
\eps \mathfrak q_\eps(z;\one,\one)
= \widetilde{\mathfrak b}_{\eps,z}(U_{\eps,z}^1,U_{\eps,z}^1) -  
\frac{1}{\rho_-}\left\langle
\Lambda_R^{\rm out}(z/c_-)\gamma_RU^1_{\eps,z},
\gamma_RU^1_{\eps,z}
\right\rangle_{H^{-1/2}(\partial B_R),H^{1/2}(\partial B_R)}.
\end{align*}
Those two identities, together with \eqref{eq:51}, \eqref{eq:52} and \eqref{eq:53} give \eqref{eq:55}. The proof of this lemma is thus completed.
\end{proof}

\begin{lemma}
\label{lem:closest-point-localization}
For each $\varepsilon>0$, let \(e_\varepsilon \in \partial D_{\varepsilon}\) and \(p_\varepsilon \in \Gamma\) be the minimizing pair satisfying \eqref{eq:mini}, and let \(\zeta_\varepsilon\) be given by \eqref{eq:scaled}. Under Assumption \ref{as_1}(i), we have that \eqref{eq:30} holds.
\end{lemma}

\begin{proof}
Since \(B\) is bounded, with the aid of \eqref{eq:scaled}, we have
\begin{align}
|p_\varepsilon-p_*|
&\leq
|p_\varepsilon-e_\varepsilon|
+
|e_\varepsilon-x_\varepsilon|
+
|x_\varepsilon-p_*|\notag
\\
&\leq
d_\varepsilon + C_B\varepsilon+|x_\varepsilon-p_*|. \notag
\end{align}
This, together with the assumption \eqref{eq:asump_1}
implies that
\begin{align*}
\lim_{\eps \rightarrow 0} p_\varepsilon = p_*.
\end{align*}
The continuity of the unit normal field on the \(C^2\) surface
\(\Gamma\) therefore gives
\begin{align*}
\lim_{\eps \rightarrow 0}\nu_\Gamma(p_\varepsilon)
=
\nu_\Gamma(p_*)
=
\nu_*.
\end{align*}

It remains to identify the limit of \(\xi_\varepsilon\). Since
\(\overline B\) is compact, every sequence
\(\{\xi_\varepsilon\}\subset\overline B\) admits a convergent
subsequence. Consider an arbitrary subsequence, still denoted by
\(\{\xi_\varepsilon\}\), such that
\begin{align*}
\lim_{\eps \rightarrow 0}\xi_\varepsilon = \overline\xi
\in\overline B.
\end{align*}
Fix \(\eta \in \overline{B}\). In the rest of the proof, we assume that \(\varepsilon\) is sufficiently small. Since \(s_\Gamma\in C^2(U_{r_\Gamma})\), using \eqref{eq:scaled} and applying Taylor's formula at \(e_\varepsilon\) gives
\begin{align}\label{eq:4}
s_\Gamma(x_\varepsilon+\varepsilon\eta)
&=
s_\Gamma(e_\varepsilon)
+
\varepsilon
\nabla s_\Gamma(e_\varepsilon)
\cdot(\eta-\xi_\varepsilon)
+
O(\varepsilon^2), \qquad \mathrm{as}\; \eps \rightarrow 0.
\end{align}
The remainder is uniform for
\(\eta,\xi_\varepsilon\in\overline B\), because \(B\) is bounded and
\(s_\Gamma\) has bounded second derivatives in a smaller tubular
neighborhood of \(p_*\). Furthermore, since \(e_\varepsilon\) realizes the distance from
\(\overline{D_\varepsilon}\) to \(\Gamma\), we have
\begin{align}\label{eq:14}
s_\Gamma(x_\varepsilon+\varepsilon\xi_\varepsilon)
\geq
d_\varepsilon
=
s_\Gamma(e_\varepsilon).
\end{align}
Using \eqref{eq:14}, and passing to the limit on the both sides of \eqref{eq:4} along the chosen subsequence yields
\[
-\nu_*\cdot(\eta-\overline\xi)\geq0
\qquad
\text{for every }\eta\in\overline B.
\]
Therefore,
\[
\overline\xi
\in
\operatorname*{arg\,max}_{\xi\in\overline B}
\nu_*\cdot\xi.
\]
By the unique-supporting-point assumption, we have
\(\overline\xi=\xi_*\).
Therefore, every convergent subsequence of \(\{\xi_\varepsilon\}\) has the
same limit \(\xi_*\). The proof of this lemma is thus completed.
\end{proof}

We proceed to introduce the following continuity result for the capacitary
minimum under almost-everywhere convergence of uniformly elliptic
coefficients. Its proof, based on Mosco's stability theorem for
variational inequalities
\cite[Chapter~4, Theorem~4.1]{R_87},
is given in Appendix~\ref{app:D2}.

\begin{proposition}
\label{lem:weighted-capacity-continuity}

Assume that
\begin{align}\label{eq:15}
0<a_*\leq b_n({\widetilde y}),\; b({\widetilde y})\leq a^*<\infty
\quad\text{for almost every }{\widetilde y}\in\mathbb R^3,
\end{align}
and that 
\begin{align}\label{eq:16}
\lim_{n \rightarrow \infty}b_n({\widetilde y}) = b({\widetilde y})\quad \mathrm{for}\;\mathrm{almost}\; \mathrm{every}\;{\widetilde y}\in\mathbb R^3.
\end{align}
Then, we have 
\begin{align}\label{eq:12}
\lim_{n\rightarrow \infty}m(b_n) = m(b).
\end{align}
Here, 
\[
m(e) := \inf_{v\in\mathcal A_B}
\int_{\mathbb R^3\setminus\overline B}
e({\widetilde y})|\nabla v({\widetilde y})|^2\,d{\widetilde y}
\]
for each uniformly positive coefficient \(e \in L^{\infty}(\R^3)\).
\end{proposition}

\begin{lemma}
\label{le:static-capacitance}
We have that 
in the finite scaled-distance regime,
\[
\lim_{\eps \rightarrow 0} \mathfrak q_\varepsilon(0;1,1)
=
\icap^{\mathrm{flat}}_{\Theta_*,\nu_*}(B),
\qquad
\nu_*:=\nu_\Gamma(p_*),
\]
whereas, if \(\Theta_*=\infty\),
\[
\lim_{\eps\rightarrow 0} \mathfrak q_\varepsilon(0;1,1)
=
\rho_+^{-1}\icap_{\mathrm{Newt}}(B).
\]
\end{lemma}

\begin{proof}
Let $P_\eps$ denote the static capacitary potential of $D_\eps$ in
the background, as stated in Lemma \ref{le:main-D2}. Introducing a change of variable $ x = x_\eps + \eps {\widetilde x}$, we have 
\[
\int_{\R^3\setminus \overline{D_\eps}} \frac{1}{\rho^{(0)}(x)}|\nabla P_\eps(x)|^2\,\dd x = \eps \int_{\R^3\setminus \overline{D_\eps}} \frac{1}{\rho^{(0)}(x_\eps + \eps {\widetilde x})}|\nabla_{\widetilde x} P_\eps(x_\eps + \eps {\widetilde x})|^2\,\dd {\widetilde x}. 
\]
With the aid of the representation \eqref{eq:45}, we have 
\begin{align} \label{eq:24}
\mathfrak q_\eps(0,1,1)
=
\inf_{v\in\mathcal A_B}
\int_{\mathbb R^3\setminus\overline B}
\widehat a_\varepsilon({\widetilde x})|\nabla v({\widetilde x})|^2\,d{\widetilde x}
=:
m(\widehat a_\varepsilon).
\end{align}
Here,
\[
\widehat a_\varepsilon({\widetilde y})
:=
\frac{1}{\rho^{(0)}\bigl(x_\eps + \eps {\widetilde y}\bigr)}, \quad {\widetilde y} \in \R^3 \backslash B.
\]

In particular,
\begin{align}\label{eq:dist}
s_\Gamma=0
\quad\text{and}\quad
\nabla s_\Gamma = -\nu_\Gamma
\qquad\text{on}\;\Gamma.
\end{align}
Let \(e_\varepsilon \in \partial D_{\varepsilon}\) and \(p_\varepsilon \in \Gamma\) be the minimizing pair satisfying \eqref{eq:mini}, and let \(\zeta_\varepsilon\) be given by \eqref{eq:scaled}. Then, applying \eqref{eq:mini}, \eqref{eq:21} and \eqref{eq:2}, we have 
\begin{align*}
e_\varepsilon
= p_\eps
-
\eps \Theta_\eps \nu_\Gamma\bigl(p_\eps).   
\end{align*}
This, together with \eqref{eq:scaled} gives
\begin{align}\label{eq:20}
    x_\eps + \eps {\widetilde y} = p_\eps - \eps \Theta_\eps\nu_\Gamma\bigl(p_\eps) + \eps ({\widetilde y}-\xi_\eps).
\end{align}

Given any \({\widetilde y} \in \R^3\), we can find \(\eps_{\widetilde y}>0\) such that 
\begin{align*}
T_\eps({\widetilde y}) \in U_{r_\Gamma},\qquad \mathrm{for}\; \eps \in (0,\eps_{\widetilde y}).
\end{align*}
Since \(s_{\Gamma} \in C^2(U_{r_\Gamma})\), using \eqref{eq:dist} and \eqref{eq:20}, we have 
\begin{align}\label{eq:23}
s_\Gamma(x_\eps + \eps {\widetilde y})
&=
-\varepsilon
\nu_\Gamma(p_\varepsilon)
\cdot(-\Theta_\eps \nu_\Gamma\bigl(p_\eps) + {\widetilde y}-\xi_\varepsilon)
+
O(\varepsilon^2), \qquad \mathrm{as}\; \eps \rightarrow 0.
\end{align}
Since
\[
\rho^{(0)}(x) = \begin{cases}
    \rho_+, & \mathrm{if}\; s_\Gamma(x) > 0 ,\\
    \rho_-, & \mathrm{if}\; s_\Gamma(x) < 0,
\end{cases}
\]
it can be deduced from \eqref{eq:23} that for sufficiently small $\eps > 0$, 
\begin{align*}
    \widehat a_\varepsilon({\widetilde y}) = \begin{cases}
    \frac 1{\rho_+}, & \mathrm{if}\; 
\nu_\Gamma(p_\varepsilon)
\cdot(\Theta_\eps \nu_\Gamma\bigl(p_\eps) - {\widetilde y} + \xi_\varepsilon) > 0 ,\\
    \frac{1}{\rho_-}, & \mathrm{if}\;
\nu_\Gamma(p_\varepsilon)
\cdot(\Theta_\eps \nu_\Gamma\bigl(p_\eps) - {\widetilde y} + \xi_\varepsilon) < 0.
    \end{cases}
\end{align*}
Combining this with Lemma \ref{lem:closest-point-localization} gives 
\[
\lim_{\eps \rightarrow 0
}\widehat a_\varepsilon({\widetilde y}) = \begin{cases} 
    a^{\mathrm{flat}}_{\Theta_*,\nu_*}(B), & \mathrm{if}\; \Theta_*\; \textrm{is finite},\\
a_{\infty,\nu}^{\mathrm{flat}}({\widetilde y}), & \mathrm{if}\; \Theta_*=\infty.
\end{cases}
\]
Here, \(a^{\mathrm{flat}}_{\Theta_*,\nu_*}\) and \(a_{\infty,\nu}^{\mathrm{flat}}\) are specified by \eqref{eq:flat_1} and \eqref{eq:flat_2}, respectively.

Using Lemma \ref{lem:weighted-capacity-continuity}, we obtain 
\begin{align}\label{eq:26}
\lim_{\eps \rightarrow 0
}m(\widehat a_\varepsilon({\widetilde y})) = \begin{cases} 
    \icap^{\mathrm{flat}}_{\Theta_*,\nu_*}(B), & \mathrm{if}\; \Theta_*\; \textrm{is finite},\\
\icap_{\infty,\nu}^{\mathrm{flat}}({\widetilde y}), & \mathrm{if}\; \Theta_*=\infty.
\end{cases}
\end{align}
Therefore, the assertion of this lemma follows from \eqref{eq:24} and \eqref{eq:26}.
\end{proof}

Now, we are ready to give the proof Theorem \ref{th:main}.

\begin{proof} [Proof of Theorem \ref{th:main}]
Let $\eps$ be sufficiently small throughout the proof. It follows from \eqref{eq:variational_1} in Lemma \ref{le:1} and \eqref{eq:22} that 
\begin{align*}
    \mathfrak A_\eps(z) = -
\frac{z^2}{\bar k_b} |B| +
\mathfrak q_\eps(z;1,1).
\end{align*}
In conjunction with Lemma \ref{le:exact-schur} and statement \eqref{d2} of Lemma \ref{le:main-D2}, we obtain that the scattering resonances in $K$ are the zeros of the following equation
\begin{align}\label{eq:27}
 - \frac{z^2}{\bar k_b} |B| +
\mathfrak q_\eps(0;1,1) + \mathrm{Res}_\eps(z) = 0.
\end{align}
Here, \(\mathrm{Res}_\eps\) is analytic in $K$ and satisfies 
\begin{align} \label{eq:13}
|\mathrm{Res}_\eps(z)| \le C_K\eps, \qquad z \in K.
\end{align}
Set 
\begin{align*}
   w_M:= \begin{cases} 
    \overline k_b\icap^{\mathrm{flat}}_{\Theta_*,\nu_*}(B)/|B|, & \mathrm{if}\; \Theta_*\; \textrm{is finite},\\
\overline k_b\icap_{\infty,\nu}^{\mathrm{flat}}({\widetilde x})/|B|, & \mathrm{if}\; \Theta_*=\infty.
\end{cases} 
\end{align*}
Choose \(0<\delta< \sqrt {w_M}/2\) such that
the disjoint discs \(B_\delta(\pm \sqrt{ w_M})\) have closures
contained in \(\mathcal O\). On their boundaries, using Lemma \ref{le:static-capacitance} and \eqref{eq:13}, we have
\begin{align*}
    \left|z^2-w_M\right|\ge\delta(2\sqrt{w_M}-\delta)>0,\qquad  \lim_{\eps \rightarrow 0} \frac{\overline k_b}{|B|}(\mathrm{Res}_\eps(z) + \mathfrak q_\eps(z,1,1))-w_M = 0.
\end{align*}
Using Rouch\'e's theorem, there exists $\eps_K$ such that for $\eps \in (0,\eps_K)$, each disc contains exactly one zero
of \(F_\varepsilon\), counted with multiplicity.
These zeros, denoted by \(z_\varepsilon^\pm\), are thus
simple and give two simple resonances.
Moreover,
factoring the left-hand side of \eqref{eq:27} and using the separation of
the two roots yields \eqref{eq:main_1} with the remainder term satisfying \eqref{eq:main_2}.

We note that \(K\Subset\mathcal O\) is compact.
The polynomial \(z^2-w_M\) is bounded away from zero on
\(K\setminus(B_\delta(+\sqrt {w_M})\cup B_\delta(-\sqrt {w_M}))\). Applying Lemma \ref{le:static-capacitance} and \eqref{eq:13} again,
\eqref{eq:27} has no zeros on this set for sufficiently
small \(\varepsilon\). There are
no resonances there.
Consequently, \(K\) contains at most these two simple
resonances. If both \(\pm\sqrt{w_M}\) belong to the interior
of \(K\), then that \(K\) contains exactly two Minnaert resonances for all sufficiently small \(\varepsilon\).

Furthermore, with the aid of \eqref{eq:45}, we have
\[
\mathfrak q_\varepsilon(0;1,1)
= \mathfrak C_\eps.
\]
The asymptotics of the normalized physical capacitance $\mathfrak C_\eps$ follows from Lemma \eqref{le:static-capacitance}.

\end{proof}

\begin{appendices}

\section{Flat-interface capacitance: representation and monotonicity} \label{app:flat-capacitance}

\begin{proposition} \label{pr:fc}
Let \(\Theta\in[0,+\infty)\) and \(\nu\in\mathbb S^2\).
The flat-interface capacitance
\(\icap^{\mathrm{flat}}_{\Theta,\nu}(B)\)
has the following properties.

\begin{enumerate}[(a)]
    \item \label{f1}
    For fixed \(B\) and \(\nu\), the flat-interface capacitance
\(\icap^{\mathrm{flat}}_{\Theta,\nu}(B)\)
    is nondecreasing if \(\rho_+ < \rho_-\), nonincreasing if
    \(\rho_+ > \rho_-\), and constant if \(\rho_+=\rho_-\).
    \item \label{f2}
The flat-interface capacitance
\(\icap^{\mathrm{flat}}_{\Theta,\nu}(B)\) admits the boundary-integral representation \eqref{eq:flat-capacity-layer}.   
\end{enumerate}
\end{proposition}

\begin{proof}
\eqref{f1}
Let \(0\leq\Theta_1<\Theta_2\). A straightforward calculation gives
\[
a^{\mathrm{flat}}_{\Theta_2,\nu}-a^{\mathrm{flat}}_{\Theta_1,\nu}
=
\left(\frac 1{\rho_+} - \frac{1}{\rho_-}\right)
\mathbf 1_{\{ \Theta_1 < \nu \cdot {\widetilde x} - l_B(\nu) < \Theta_2\}}
\qquad\text{in }\mathbb R^3.
\]
Consequently, for every \(v\in\mathcal A_B\),
\begin{align*}
\int_{\mathbb R^3\setminus\overline B}
a^{\mathrm{flat}}_{\Theta_2,\nu}({\widetilde x})|\nabla v({\widetilde x})|^2\,d{\widetilde x} - \int_{\mathbb R^3\setminus\overline B}
a^{\mathrm{flat}}_{\Theta_1,\nu}({\widetilde x})|\nabla v({\widetilde x})|^2\,d{\widetilde x}\\
=
\left(\frac 1{\rho_+} - \frac{1}{\rho_-}\right)
\int_{\{\Theta_1 < \nu \cdot {\widetilde x} - l_B(\nu) < \Theta_2\}}
|\nabla v({\widetilde x})|^2\,d{\widetilde x}.
\end{align*}
Therefore, we arrive at 
\begin{align*}
&\icap^{\mathrm{flat}}_{\Theta_2,\nu}(B)
\le
\icap^{\mathrm{flat}}_{\Theta_1,\nu}(B),\qquad \mathrm{if}\; \rho_+ \ge \rho_-,\\
&\icap^{\mathrm{flat}}_{\Theta_2,\nu}(B)
\ge 
\icap^{\mathrm{flat}}_{\Theta_1,\nu}(B),\qquad \mathrm{if}\;\rho_+ < \rho_-.
\end{align*}
This yields the statement.

\eqref{f2}
Let \(\mathcal N\) denote the Newtonian single-layer potential,
\[
    (\mathcal N\varphi)({\widetilde x})
    :=
    \frac{1}{4\pi}
    \int_{\partial B}\frac{\varphi({\widetilde y})}{|{\widetilde x}-{\widetilde y}|}\,dS({\widetilde y}).
\]
Since
\(\mathcal N\varphi\in  D^{1,2}(\mathbb R^3)\) and
\begin{align} \label{eq:54}
    \int_{\mathbb R^3}
        \left[\nabla\mathcal N\varphi\right]({\widetilde x})\cdot \overline{\nabla v}({\widetilde x})\,d{\widetilde x}
    =
    \int_{\partial B}\varphi({\widetilde x})(\gamma_{\partial B}\overline v)({\widetilde x}) d\sigma({\widetilde x}),
    \qquad v\in D^{1,2}(\mathbb R^3).
\end{align}
This, together with the trace theorem and a bounded right inverse
of the trace with compact support gives
\begin{align}\label{eq:free-layer-energy}
    C_1\|\varphi\|_{H^{-1/2}(\partial B)}^2
    \le
    \|\nabla\mathcal N\varphi\|_{L^2(\mathbb R^3)}^2
    \le
    C_2\|\varphi\|_{H^{-1/2}(\partial B)}^2.
\end{align}
Furthermore, a straightforward calculation gives
\begin{align} \label{eq:57}
    \left|
        \int_{\mathbb R^3}\left[
        \nabla\mathcal N\varphi\cdot
        \overline{\nabla\bigl((\mathcal N\varphi)\circ R_\Theta\bigr)}\right]({\widetilde x})
        \,d{\widetilde x}
    \right| \le
    \|\nabla\mathcal N\varphi\|_{L^2(\mathbb R^3)}^2.
\end{align}
Here, the reflection operator \(R_\Theta\) is specified by \eqref{eq:reflect}.
It follows from \eqref{eq:54} that 
\begin{align} \label{eq:25}
\frac{1}{\rho_+}\langle \varphi, \mathcal S_{\Theta,\nu}\varphi
    \rangle_{(H^{-1/2}(\partial B),H^{1/2}(\partial B))} & = \int_{\mathbb R^3}
        \left[\nabla\mathcal N\varphi\right]({\widetilde x})\cdot \overline{\nabla (\mathcal N\varphi)}({\widetilde x})\,d{\widetilde x} \notag\\
&  \frac{1/\rho_+ -  1/\rho_-}{1/\rho_+ + 1/\rho_-} \int_{\mathbb R^3}
        \left[\nabla\mathcal N\varphi\right]({\widetilde x})\cdot \overline{\nabla (\mathcal N \varphi \circ R_\Theta)}({\widetilde x})\,d{\widetilde x}.
\end{align}
In conjunction with \eqref{eq:free-layer-energy} and \eqref{eq:57} gives
\begin{align*}
    \langle\varphi,\mathcal S_{\Theta,\nu}\varphi
    \rangle_{H^{-1/2}(\partial B),H^{1/2}(\partial B)}
    &\ge
    C
    \|\nabla\mathcal N\varphi\|_{L^2(\mathbb R^3)}^2 \ge
    C\|\varphi\|_{H^{-1/2}(\partial B)}^2,
\end{align*}
The Lax--Milgram theorem therefore shows that
\(\mathcal S_{\Theta,\nu}\) is an isomorphism and that
\begin{equation}\label{eq:flat-layer-uniform-inverse}
    \sup_{\Theta\ge0}
    \|\mathcal S_{\Theta,\nu}^{-1}\|_{
        \mathcal L(H^{1/2}(\partial B),H^{-1/2}(\partial B))}
    <\infty.
\end{equation}

First, we show that \eqref{eq:flat-capacity-layer} holds for $\Theta > 0 $. Set
\(\varphi_\Theta:=\mathcal S_{\Theta,\nu}^{-1}1\) and 
define the transmission potential
\[
    W_\Theta({\widetilde x})
    :=
    \begin{cases}
        \displaystyle
        \rho_+
        \bigl(
            \mathcal N \varphi_\Theta({\widetilde x})+ \frac{1/\rho_+ -  1/\rho_-}{1/\rho_+ + 1/\rho_-} \mathcal N \varphi_\Theta(R_\Theta {\widetilde x})
        \bigr),
        & {\widetilde x}\in \left\{{\widetilde x}\in\R^3: \nu\cdot {\widetilde x} < \ell_B(\nu)+\Theta\right\}, \\[6pt]
        \displaystyle
        \frac{2}{1/{\rho_+} + 1/{\rho_-}}\,\mathcal N \varphi_\Theta({\widetilde x}),
        & {\widetilde x}\in \left\{{\widetilde x}\in\R^3: \nu\cdot {\widetilde x}>\ell_B(\nu)+\Theta\right\}.
    \end{cases}
\]
Clearly,
\(W_\Theta\in  D^{1,2}(\mathbb R^3)\) and
\begin{equation}\label{eq:flat-layer-weak-identity}
    \int_{\mathbb R^3}
        a^{\mathrm{flat}}_{\Theta,\nu}({\widetilde x})\nabla W_\Theta({\widetilde x})\cdot \overline{\nabla v}({\widetilde x})\,d{\widetilde x}
    = \int_{\partial B}\varphi_{\Theta}({\widetilde x})(\gamma_{\partial B} \overline v)({\widetilde x}) dS({\widetilde x}),
    \qquad v\in D^{1,2}(\mathbb R^3).
\end{equation}
Moreover, it can be seen that
\(W_\Theta=1\) in \(\overline B\). Building upon this and
\eqref{eq:flat-layer-weak-identity}, we readily obtain that $W_\Theta$ is the capacity potential, that is,
\begin{align*}
    \icap^{\mathrm{flat}}_{\Theta,\nu}(B)
    =
    \int_{\mathbb R^3}
        a^{\mathrm{flat}}_{\Theta,\nu}(\widetilde x)|\nabla W_\Theta(\widetilde x)|^2\,d{\widetilde x}
    = \int_{\partial B}\varphi_{\Theta}({\widetilde x})dS({\widetilde x}),
    \qquad \Theta>0,
\end{align*}
which implies \eqref{eq:flat-capacity-layer} for \(\Theta > 0\).

Second, we show that \eqref{eq:flat-capacity-layer} holds for \(\Theta = 0\). Due to the continuity of translations in \(D^{1,2}(\mathbb R^3)\), for every fixed \(\varphi\in H^{-1/2}(\partial B)\), it can be deduced from \eqref{eq:25} that 
\[
    \lim_{\Theta \rightarrow 0}\mathcal S_{\Theta,\nu}\varphi
    =
    \mathcal S_{0,\nu}\varphi.
\]
This, together with \eqref{eq:flat-layer-uniform-inverse} gives
\begin{align}\label{eq:40}
 \lim_{\Theta \rightarrow 0}\mathcal S^{-1}_{\Theta,\nu}\varphi
    =
    \mathcal S^{-1}_{0,\nu}\varphi.
\end{align}
Furthermore, since 
\begin{align*}
    \lim_{\Theta \rightarrow 0} a^{\mathrm{flat}}_{\Theta,\nu}({\widetilde y})= a^{\mathrm{flat}}_{0,\nu}({\widetilde y}) \quad \textrm{for almost every } {\widetilde y} \in \R^3,
\end{align*}
it follows from Proposition \ref{lem:weighted-capacity-continuity} that
\begin{align*}
\lim_{\Theta \rightarrow 0}\icap^{\mathrm{flat}}_{\Theta,\nu}(B) = \icap^{\mathrm{flat}}_{0,\nu}(B).
\end{align*}
Combining this with \eqref{eq:40} gives \eqref{eq:flat-capacity-layer} for $\Theta = 0$.
\end{proof}

\section{Proof of Proposition \ref{le:uniform-inverse}}
\label{app:uniform-stability}

This section is devoted to proving Proposition \ref{le:uniform-inverse}. Before proving this, we prepare two lemmas.

\begin{lemma}
Let $\eps >0$ and let the sesquilinear form \(\mathfrak b_{\eps,z}\) be specified by \eqref{eq:truncated-form}.
For every compact
$K\subset\mathcal O$, there exists $C_K>0$ such that
\begin{equation}
\label{eq:uniform-garding}
\|v\|_{H^1(B_R\backslash \overline{D_\eps})}^2
\le C_K\left(
|\mathfrak b_{\eps,z}(v,v)|+\|v\|_{L^2(B_R\backslash \overline{D_\eps})}^2
\right),
\qquad v\in V_\eps, 
\end{equation}
uniformly in $\eps$ and $z\in K$.

Furthermore, for every fixed $\eps\ge0$, $\mathfrak A_\eps(z)$ is Fredholm of index zero.
\end{lemma}

\begin{proof}
With the aid of the spherical-harmonic representation of the exterior Dirichlet-to-Neumann map, we have 
\[
\Lambda_R^{\rm out}(z){\widetilde y}_\ell^m
=
\lambda_\ell(z){\widetilde y}_\ell^m,
\qquad
\lambda_\ell(z)
=
z\frac{(h_\ell^{(1)})'(zR)}
        {h_\ell^{(1)}(zR)} .
\]
The singularity at zero is removable, with
\[
\lambda_\ell(0)=-\frac{\ell+1}R, \qquad \ell \in \mathbb Z_0.
\]
For \(\ell\geq1\), the recurrence relations
\cite[Eqs.~(10.51.1)--(10.51.2)]{DLMF} yield
\begin{align*}
\lambda_\ell(z)
=
-\frac{\ell+1}{R}
+
\frac{r_\ell(zR)}{R}, \qquad
    r_\ell(t):=\frac{t h_{\ell-1}^{(1)}(t)}
                         {h_\ell^{(1)}(t)}.
\end{align*}
These functions satisfy
\[
r_1(t)=\frac{t^2}{1-\mathrm{i}t},
\qquad
r_\ell(t)
=
\frac{t^2}{2\ell-1-r_{\ell-1}(t)},
\quad \ell\geq2.
\]
Then, an induction shows that
\[
|r_\ell(t)|
\leq
\frac{2|t|^2}{2\ell-1},
\qquad
|t|\leq\frac12,\quad \ell\geq1,
\]
where the quotients are understood by their removable continuation
at zero. Away from zero, the finite-sum representation
\cite[Eqs.~(10.49.1), (10.49.6)]{DLMF} gives \(r_\ell(zR)=O_K(\ell^{-1})\) uniformly for
\(z\in K\) with \(|z|\geq(2R)^{-1}\).
\[ 
\lambda_\ell(z)
=
-\frac{\ell+1}{R}
+
O_{K}\bigl((\ell+1)^{-1}\bigr), \qquad \textrm{uniformly on } K. 
\]
Thus \(\operatorname{Re}\lambda_\ell(z)\leq0\) for all
sufficiently large \(\ell\), uniformly on \(K\).
The remaining finitely many eigenvalues are bounded on \(K\),
since \(K\) avoids the poles.
Summing over the spherical harmonics therefore gives
\begin{align}\label{eq:Gar}
-\operatorname{Re}
\left\langle
\Lambda_R^{\mathrm{out}}(z)\varphi,\varphi
\right\rangle_{H^{-1/2}(\partial B_R), \; H^{1/2}(\partial B_R)}
\geq
-C_K\|\varphi\|_{L^2(\partial B_R)}^2,
\qquad z\in K.
\end{align}

Furthermore, since $R$ is fixed, using the multiplicative trace inequality, we have 
\[
\|\gamma_R v\|_{L^2(\partial B_R)}^2
\le
C\|v\|_{L^2(\mathcal C_R)}\|v\|_{H^1(\mathcal C_R)}.
\]
Here, the collar $\mathcal C_R$ is defined by 
\[
\mathcal C_R:=B_R\setminus\overline{B_{R-1}}.
\]
Due to the fact that \(D_\varepsilon\) is contained in \(B_{R-2}\), $\mathcal C_R$
is contained in \(B_R \backslash \overline{D_\eps}\) for every \(\varepsilon>0\).
In conjunction with Young's inequality, for every \(\eta>0\), we obtain
\begin{align}\label{eq:trace}
\|\gamma_R v\|_{L^2(\partial B_R)}^2
\le
\eta\|\nabla v\|_{L^2(\mathcal C_R)}^2
+
C_\eta\|v\|_{L^2(\mathcal C_R)}^2.
\end{align}
It follows from \eqref{eq:Gar} and \eqref{eq:trace} that 
\begin{align}
\operatorname{Re}\mathfrak b_{\varepsilon,z}(v,v)
&\geq \min(\rho^{-1}_+,\rho^{-1}_-)\|\nabla v\|_{L^2(B_R \backslash \overline{D_\eps})}^2
 - M_K\|v\|_{L^2(B_R \backslash \overline{D_\eps})}^2
 - C_K\|\gamma_Rv\|_{L^2(\partial B_R)}^2                                      \notag\\
&\geq
 \bigl(\min(\rho^{-1}_+,\rho^{-1}_-) - C_K\eta\bigr)
 \|\nabla v\|_{L^2(B_R \backslash \overline{D_\eps})}^2
 -
 \bigl(M_K + C_KC_\eta\bigr)
 \|v\|_{L^2(B_R \backslash \overline{D_\eps})}^2. \label{eq:36}
\end{align}
Here,
\[
 M_K:=\sup_{z\in K}|z|^2\|k_0^{-1}\|_{L^\infty(B_R)}.
\]
Choosing \(\eta>0\) so that 
\[
C_K\eta\leq \frac{\min(\rho^{-1}_+,\rho^{-1}_-)}2,
\]
we find
\begin{align}\label{eq:37}
 \operatorname{Re}\mathfrak b_{\varepsilon,z}(v,v)
 \geq
 \frac{\min(\rho^{-1}_+,\rho^{-1}_-)}{2}\|\nabla v\|_{L^2(B_R \backslash \overline{D_\eps})}^2
 -
 C'_K\|v\|_{L^2(B_R \backslash \overline{D_\eps})}^2 ,
\end{align}
uniformly in \(\varepsilon\) and \(z\in K\). Consequently, combining \eqref{eq:36} and \eqref{eq:37} gives \eqref{eq:uniform-garding}.

It remains to prove the Fredholm property. Fix
\(\varepsilon\geq0\) and \(z\in U\). By the preceding estimate, one can
choose \(\mu>0\) such that the form
\[
 \mathfrak c_{\varepsilon,z}(u,v)
 :=
 \mathfrak b_{\varepsilon,z}(u,v)
 +
 \mu\int_{B_R \backslash \overline{D_\eps}}u(x)\overline v(x)\,\mathrm dx
\]
is coercive on \(V_\varepsilon\). Hence, by the Lax--Milgram theorem,
the induced operator
\( \mathfrak C_\varepsilon(z):V_\varepsilon\longrightarrow V_\varepsilon\) is an isomorphism. Since  
\begin{align}\label{eq:35}
 \mathfrak A_\eps(z)(z)
 =
 \mathfrak C_\varepsilon(z)-\mu\mathfrak J_\varepsilon ,
\end{align}
where \(\mathfrak J_\varepsilon:V_\varepsilon\to V_\varepsilon\) is
defined by
\[
 \bigl\langle\mathfrak J_\varepsilon u,v\bigr\rangle
 :=
 \int_{B_R\backslash \overline D_\eps}u(x)\overline v(x)\,\mathrm dx .
\]
It follows from the zero-extension map
\(\mathcal E_\varepsilon:V_\varepsilon\to H^1(B_R)\) and the compact embedding from
\(H^1(B_R)\) into \(L^2(B_R)\) that 
\(\mathfrak J_\varepsilon\) is compact. Combining this with the decomposition \eqref{eq:35} and the invertibility of \(\mathfrak C_\eps\) gives that \(\mathfrak A_\eps(z)\) is Fredholm of
index zero.
\end{proof}

\begin{lemma}
\label{lem:mosco-hole}
Let $\eps >0$. For each \(v\in H^1(B_R)\), there exists $w_\eps \in V_\eps$ such that 
\begin{align}\label{eq:19}
    \lim_{\eps \rightarrow 0} \mathcal E_\eps w_\eps = v.
\end{align}
Here \(\mathcal E_\eps\) denotes the zero extension through \(D_\eps\).
\end{lemma}

\begin{proof}
Fix \(R'>R\). Given each \(v\in H^1(B_R)\), we extend it to a function
\(\widetilde v\in H_0^1(B_{R'})\). Since \(\overline{D_\varepsilon}\) is contained in a ball
of radius \(C\varepsilon\), its capacity in the sense of
\cite[Section~5.3]{Daners_08} tends to zero. Applying \cite[Proposition~5.3.3]{Daners_08} to
$B_{R'}\setminus\overline{D_\varepsilon}$ and $B_{R'}$,
we obtain the strong approximation property stated in
\cite[Assumption~5.2.2]{Daners_08}.
Hence, there exist
\(\widetilde w_\varepsilon\in
H_0^1(B_{R'}\setminus\overline{D_\varepsilon})\)
whose zero extensions \(\widehat w_\varepsilon\) satisfy that
\begin{align}
   \lim_{n\rightarrow \infty} \widehat w_\varepsilon = \widetilde v
    \qquad \textrm{in }H^1(B_{R'}). \notag
\end{align}
Setting
\(w_\varepsilon
:=\widetilde w_\varepsilon|_{B_R\setminus\overline{D_\varepsilon}}\),
we have \(w_\varepsilon\in V_\varepsilon\). Since
\[
 \|\mathcal E_\varepsilon w_\varepsilon-v\|_{H^1(B_R)}
    \leq 
    \|\widehat w_\varepsilon-\widetilde v\|_{H^1(B_{R'})},
\]
we readily obtain \eqref{eq:19}. The proof of this lemma is thus completed.
\end{proof}

We are ready to give the proof of Proposition \ref{le:uniform-inverse}.

\begin{proof}[Proof of Proposition \ref{le:uniform-inverse}]

Suppose, to the contrary, that the conclusion fails. Then, for every
\(\delta>0\) and \(M>0\), there exist
\(
0<\varepsilon(\delta)<\delta
\) and \(z\in K\)
such that either \(\mathfrak A_{\eps_\delta}(z)\) is not invertible,
or it is invertible and
\[
\left\|\mathfrak A_{{\eps_\delta}}(z)^{-1}\right\|_
{\mathcal L(V_\varepsilon,V_\varepsilon)}
>M.
\]
Taking \(\delta=n^{-1}\) and \(M=n\), we obtain sequences
\[
0<\varepsilon_n<n^{-1},
\qquad z_n\in K,
\]
for which one of these two alternatives holds.

We construct \(v_n\in V_{\varepsilon_n}\) as follows.
If \(\mathfrak A_{\varepsilon_n}(z_n)\) is not invertible, its
Fredholm property and index zero imply that its kernel is nontrivial.
We therefore choose
\[
v_n\in\ker\mathfrak A_{\varepsilon_n}(z_n),
\qquad
\|v_n\|_{V_{\varepsilon_n}}=1.
\]
Otherwise,
\[
\left\|\mathfrak A_{\varepsilon_n}(z_n)^{-1}\right\|_
{\mathcal L(V_{\varepsilon_n},V_{\varepsilon_n})}
>n.
\]
By the definition of the operator norm, there exists
\(f_n\in V_{\varepsilon_n}\) such that
\[
\|f_n\|_{V_{\varepsilon_n}}=1,
\qquad
\left\|
\mathfrak A_{\varepsilon_n}(z_n)^{-1}f_n
\right\|_{V_{\varepsilon_n}}>n.
\]
Set
\[
v_n
:=
\frac{\mathfrak A_{\varepsilon_n}(z_n)^{-1}f_n}
{\left\|
\mathfrak A_{\varepsilon_n}(z_n)^{-1}f_n
\right\|_{V_{\varepsilon_n}}}.
\]
Therefore,
\begin{align*}
\|v_n\|_{V_{\varepsilon_n}}=1,
\qquad
\lim_{n\longrightarrow 0}
\left\|
\mathfrak A_{\varepsilon_n}(z_n)v_n
\right\|_{V_{\varepsilon_n}} =0.
\end{align*}

Extend \(v_n\) by zero through \(D_{\eps_n}\). Up to a subsequence, we have
\[
v_n\rightharpoonup v\quad\text{weakly in }H^1(B_R),
\qquad
\lim_{n \rightarrow \infty} v_n = v\quad\text{in }L^2(B_R), \qquad \textrm{and } \lim_{n \rightarrow \infty} z_n = z_*\in K. 
\]
Let \(\psi\in V_0\). By Lemma \ref{lem:mosco-hole}, there exist
\(\psi_n\in V_{\eps_n}\) with
\[
\lim_{n \rightarrow \infty}\psi_n = \psi\quad\text{in }H^1(B_R).
\]
Then, it can be seen that 
\begin{align*}
\lim_{n\rightarrow \infty}\int_{B_R\setminus\overline{D_\eps}}
\bigl((\rho^0)^{-1}\nabla v_n\cdot\nabla\overline \psi_n - z^2 (k^{(0)})^{-1}\,v_n\overline \psi_n\bigr)(x)\,\dd x \\
= \int_{B_R\setminus\overline{D_\eps}}
\bigl((\rho^0)^{-1}\nabla v\cdot\nabla\overline \psi - z^2 (k^{(0)})^{-1}\,u\overline \psi \bigr)(x)\,\dd x,
\end{align*}
and that
\[
\gamma_Rv_n\rightharpoonup\gamma_Rv
\quad\text{weakly in }H^{1/2}(\partial B_R),
\qquad
\lim_{n \rightarrow \infty} \gamma_R\psi_n = \gamma_R\psi
\quad\text{in }H^{1/2}(\partial B_R),
\]
Moreover, by the holomorphy of the exterior Dirichlet-to-Neumann map, we have
\[
\lim_{n \rightarrow \infty}\Lambda_R^{\rm out}(z_n)	= \Lambda_R^{\rm out}(z_*)
\quad\text{in }\mathcal L(H^{1/2}(\partial B_R),H^{-1/2}(\partial B_R)).
\]
Based on the above discussions, we arrive at 
\[
\mathfrak b_{0,z_*}(v,\psi)=0
\qquad\text{for every }\psi\in H^1(B_R).
\]
Combining this with the hypothesis on \(K\) gives \(v=0\).
Therefore 
\begin{align*}
\lim_{n\rightarrow \infty} v_n= 0 \qquad \textrm{in } L^2(B_R). 
\end{align*}
Finally, \eqref{eq:uniform-garding} gives
\[
1=\|v_n\|_{H^1(B_R\backslash \overline{D_\eps})}^2
\le C_K\left(
|\langle\mathfrak A_{\eps_n}(z_n)v_n,v_n\rangle_{V_{\eps_n}}|
+\|v_n\|_{L^2(B_R)}^2
\right)\longrightarrow0,
\]
leading to a contradiction.
\end{proof}

\section{Proof of Proposition \ref{lem:weighted-capacity-continuity}}
\label{app:D2}

\begin{proof}[Proof of Proposition \ref{lem:weighted-capacity-continuity}]

Define \(\mathfrak B_n,\mathfrak B: \mathcal A_B \to D^{1,2}(\R^3)\) by
\[
    \langle \mathfrak B_nu,v\rangle_{D^{1,2}(\R^3)}
    :=\int_{\mathbb R^3}b_n({\widetilde y})\nabla u({\widetilde y})\cdot\nabla v({\widetilde y})\,d{\widetilde y},
    \qquad
    \langle \mathfrak B u,v\rangle_{D^{1,2}(\R^3)}
    :=\int_{\mathbb R^3} b({\widetilde y})\nabla u({\widetilde y})\cdot\nabla v({\widetilde y})\,d{\widetilde y}.
\]
With the aid of \eqref{eq:15} and \eqref{eq:16},
a straightforward calculation gives
\begin{align*}
&\|\mathfrak B_nv_n- \mathfrak B_n v\|_{D^{1,2}(\R^3)} \le a^* \|v_n - v\|_{D^{1,2}(\R^3)}, \qquad \forall v_n,\;v\in \mathcal A_B.\\
& \langle\mathfrak B_n u - \mathfrak B_n v, u - v\rangle_{D^{1,2}(\R^3)} \ge a_* \|v_n - v\|^2_{D^{1,2}(\R^3)}, \qquad \forall v_n,\;v\in D^{1,2}(\R^3).
\end{align*}
Thus, by \cite[Chapter~4, Theorem~3.4 and Remark~3.5]{R_87}, there exist minimizers \(u_n\in \mathcal A_B\) and \(u \in \mathcal A_B\) achieving \(m(b_n)\) and \(m(b)\), respectively.
solve
\begin{align*}
    \langle \mathfrak B_n u_n,v-u_n\rangle_{D^{1,2}(\R^3)} \geq 0,
    \qquad
    \langle \mathfrak B u,v-u\rangle_{D^{1,2}(\R^3)} \geq 0,
    \qquad \forall v\in \mathcal A_B.
\end{align*}

Moreover, whenever $\lim_{n\rightarrow \infty} v_n = v$, by dominated convergence theorem,
\[
   \lim_{n\rightarrow \infty} \|\mathfrak B_nv_n- \mathfrak B v\|_{D^{1,2}(\R^3)}
    \leq \lim_{n\rightarrow \infty}\left[a^*\|v_n-v\|_{D^{1,2}(\R^3)}
       +\|(b_n-b)\nabla v\|_{L^2(\mathbb R^3)}\right]
    = 0.
\]
We note that the admissible set $\mathcal A_B$ is fixed.
\cite[Chapter~4, Theorem~4.1]{R_87} yields
\begin{align*}
\lim_{n\rightarrow \infty} u_n = u.
\end{align*}
Consequently, \(\lim_{n\rightarrow \infty}\mathfrak B_n u_n =  \mathfrak B u\), and hence \eqref{eq:12} holds.
 \end{proof}

\end{appendices}

\section*{Acknowledgment}

The work of H. Diao is supported by National Natural Science Foundation of China  (Grant No. 12371422),  and the Fundamental Research Funds for the Central Universities, JLU. The work of M. Sini is supported by 2025 National Foreign Experts Program (Grant No. D20250157). The work of L. Li and M. Sini is supported by the Austrian Science Fund (FWF) grant P: 36942.

\end{document}